\documentclass[final,oneside]{amsart}

\usepackage{inputenc}
\usepackage[english]{babel}
\usepackage{graphicx}
\usepackage{enumitem}
\usepackage{dsfont}

\usepackage{color}

\usepackage[final,activate,verbose=true,auto=true]{microtype}
\usepackage[norelsize,ruled,vlined,commentsnumbered]{algorithm2e}
\usepackage{amsmath, amsfonts,amssymb,amsopn,amscd,amsthm}
\usepackage{mathrsfs}

\usepackage[textheight=235mm,textwidth=160mm,top=20mm,headsep=8mm]{geometry}
\usepackage{hyperref}
\numberwithin{equation}{section}

\newcommand{\eq}{\begin{equation}}
\newcommand{\qe}{\end{equation}}

\newcommand{\N}{\mathbb{N}}                %natural integers
\newcommand{\R}{\mathbb{R}}                     %real numbers

\newcommand{\nnd}{\mathbb{S}^{+}} 
\def\Cs{\mathscr{C}}
\def\Ms{\mathscr{M}}

\def\v{\mathbf{v}}
\def\x{x}
\def\K{\mathcal{X}}
\def\M{\mathbf{M}}
\def\B{\mathbf{B}}
\def\C{\mathbf{C}}
\def\X{\Lambda}
\def\y{\mathbf{y}}

\theoremstyle{plain}% default
\newtheorem{thm}{Theorem}[section]

\theoremstyle{definition}
\newtheorem{defn}[thm]{Definition}

\theoremstyle{remark}

\begin{document}
\sloppy

%opening
\pagestyle{headings} 
\title{D-optimal design for multivariate polynomial regression  via the Christoffel function and semidefinite relaxations} 
\date{Draft of \today}

% Authors
\author{Y. De Castro \and F. Gamboa \and D. Henrion \and R. Hess \and J.-B. Lasserre}
\address{YDC is with the Laboratoire de Math\'ematiques d'Orsay, Univ. Paris-Sud, CNRS, Universit\'e Paris-Saclay, 91405 Orsay, France.}
\email{yohann.decastro@math.u-psud.fr}
\urladdr{www.math.u-psud.fr/$\sim$decastro} 
\address{FG is with the Institut de Math\'ematiques de Toulouse (CNRS UMR 5219), Universit\'e Paul Sabatier, 118 route de Narbonne, 31062 Toulouse, France.}
\email{fabrice.gamboa@math.univ-toulouse.fr}
\urladdr{www.math.univ-toulouse.fr/$\sim$gamboa}
\address{DH is with LAAS-CNRS, Universit\'e de Toulouse, LAAS, 7 avenue du colonel Roche, F-31400 Toulouse, France
and with the Faculty of Electrical Engineering, Czech Technical University in Prague, Technick\'a 2, CZ-16626 Prague, Czech Republic}
\email{henrion@laas.fr}
\urladdr{homepages.laas.fr/henrion}
\address{RH is with LAAS-CNRS, Universit\'e de Toulouse, LAAS, 7 avenue du colonel Roche, F-31400 Toulouse, France.}
\email{rhess@laas.fr}
\address{JBL is with LAAS-CNRS, Universit\'e de Toulouse, LAAS, 7 avenue du colonel Roche, F-31400 Toulouse, France.}
\email{lasserre@laas.fr}
\urladdr{homepages.laas.fr/lasserre}

%%%%%%%%%%%%%%%%%%%%%%abstract%%%%%%%%%%%%%%%%%%%%%%%%%%%
\begin{abstract}
We present a new approach to the design of D-optimal experiments with multivariate polynomial regressions on compact semi-algebraic design spaces. We apply the moment-sum-of-squares hierarchy of semidefinite programming problems to solve numerically and approximately the optimal design problem. The geometry of the design is recovered with semidefinite programming duality theory and the Christoffel polynomial.
\end{abstract}

%%%%%%%%%%%%%Title%%%%%%%%%%%%%%%%%%%%%%%%%%%%%%%%%%%%%%%
\maketitle 

%%%%%%%%%%%%%%%%%%%%%%%%%%%%%%%%%%%%%%%%%%%%%%%%%

\section{Introduction}

\subsection{Convex design theory}
The optimum experimental designs are computational and theoretical objects that minimize the variance of the best linear unbiased estimators in regression problems. In this frame, the experimenter models the responses $z_1,\ldots,z_N$ of a random \textit{experiment} whose inputs are represented by a vector $\xi_{i}\in\R^{n}$ with respect to known \textit{regression functions} $\varphi_{1},\ldots,\varphi_{p}$, namely
\[
%\forall i\in\llbracket1,n\rrbracket,\quad 
z_{i}=\sum_{j=1}^{p}\theta_{j}\varphi_{j}(\xi_{i})+\varepsilon_{i}\,,\:\:i=1,\ldots;N
\]
where $\theta_1,\ldots,\theta_p$ are unknown parameters that the experimenter wants to estimate, $\varepsilon_{i}$ is some noise and the inputs $\xi_{i}$~are chosen by the experimenter in a \textit{design space}~$\mathcal X\subseteq\R^{n}$. Assume that the inputs $\xi_i$, for~$i=1,\ldots,N$, are chosen within a set of distinct points $x_1, \ldots, x_\ell$ with $\ell \leq N$, and let $n_k$ denote the number of times the particular point $x_k$ occurs among $\xi_1, \ldots, \xi_N$. This would be summarized by %the notation
\eq
\label{eq:defDesignExact}
\zeta:=\left(\begin{array}{ccc}x_{1} & \cdots & x_{\ell} \\ \frac{n_1}{N} & \cdots & \frac{n_{\ell}}{N}\end{array}\right)
\qe
whose first row gives the points in the design space $\mathcal X$ where the inputs parameters have to be taken and the second row indicates the experimenter which proportion of experiments (frequencies) have to be done at these points. The goal of the design of experiment theory is then to assess which input parameters and frequencies the experimenter has to consider. For a given $\zeta$ the standard analysis of the Gaussian linear model shows that the minimal covariance matrix (with respect to Loewner ordering) of unbiased estimators can be expressed in terms of the Moore-Penrose pseudoinverse of the \textit{information matrix} which is defined by
\eq
\label{eq:defInformationMatrix}
M(\zeta):=\sum_{i=1}^{\ell}w_{i}\Phi(x_{i})\Phi^{\top}(x_{i})
\qe
where $\Phi:=(\phi_1,\ldots,\phi_p)$ is the column vector of regression functions.
One major aspect of design of experiment theory seeks to maximize the information matrix over the set of all possible $\zeta$. Notice the Loewner ordering is partial and, in general, there is no greatest element among all possible information matrices $M(\zeta)$. The standard approach is then to consider some statistical criteria, namely \textit{Kiefer's $\phi_{q}$-criteria}~\cite{kiefer1974general}, in order to describe and construct the ``\textit{optimal designs}'' with respect to those criteria. Observe that the information matrix  belongs to $\nnd_{p}$, the space of symmetric nonnegative definite matrices of size $p$, and for all $q\in[-\infty,1]$ define the function
\[
\phi_{q}\,:=\,\left\{\begin{array}{cll}\nnd_{p} & \to & \R \\ M & \mapsto & \phi_{q}(M)\end{array}\right.
\]
where for positive definite matrices $M$ it holds 
\[
\phi_{q}(M):=\left\{\begin{array}{ll}(\frac1p\mathrm{trace}(M^{q}))^{1/q} & \mathrm{if}\ q\neq-\infty,0  \\ \det(M)^{1/p} & \mathrm{if}\ q=0  \\\lambda_{\min}(M) & \mathrm{if}\ q=-\infty\end{array}\right.
\]
and for nonnegative definite matrices $M$ it holds
\[
\phi_{q}(M):=\left\{\begin{array}{ll}(\frac1p\mathrm{trace}(M^{q}))^{1/q} & \mathrm{if}\ q\in(0,1]  \\ 0 & \mathrm{if}\ q\in[-\infty,0].\end{array}\right.
\]
Those criteria are meant to be real valued, positively homogeneous, non constant, upper semi-continuous, isotonic (with respect to the Loewner ordering) and concave functions. Throughout this paper, we restrict ourselves to the $D$-optimality criteria which corresponds to the choice $q=0$.
Other criteria will be studied elsewhere.

In particular, in this paper we search for solutions $\zeta^{\star}$ to the following optimization problem
\eq
\label{eq:defOptimumdesigns}
         \max \log \det M(\zeta)
\qe
where the maximum is taken over all $\zeta$ of the form \eqref{eq:defDesignExact}. Note that the logarithm of the determinant is used instead of $\phi_0$ because of its standard use in semidefinite programming (SDP) as a barrier function for the cone of positive definite matrices.

\subsection{State of the art}
Optimal design is at the heart of statistical planning for inference in the linear model, see for example \cite{BH78}. While the case of discrete input  factors is generally tackled by algebraic and combinatoric arguments (e.g., \cite{Bai2008}), the one of continuous input  factors  often leads  to an optimization problem.  In general, the continuous factors are generated by  a vector $\Phi$ of linearly independent regular functions  on  the  \textit{design space}~$\mathcal X$. 
One way to handle the problem is to focus only on~$\mathcal X$ ignoring the function $\Phi$ and to try to draw the design points {\it filling} the best the set
$\mathcal X$. This is generally done by optimizing a cost function on $\mathcal X^N$ that traduces the way the design points are positioned between each other and/or how they fill the space. Generic examples are the so-called maxmin or minmax criteria (see for example \cite{prozac88} or \cite{wawa97}) and 
the minimum discrepancy designs (see for example~\cite{Liuliu05}). Another point of view\textemdash which is the one developed here\textemdash  relies on the {\it maximization} of the information matrix. Of course, as explained before, the  Loewner order is partial and so the optimization can not stand on this matrix but on one of
its feature. A pioneer paper adopting this point of view is the one of Elfving \cite{Elf52}. In the early 60's, in a series of papers, Kiefer and Wolwofitz shade 
new lights on this kind of methods for experimental design by introducing the equivalence principle and proposing in some cases algorithms to solve the optimization problem, see \cite{kiefer1974general} and references therein. Following the early works of Karlin and Studden (\cite{KarStu66}, \cite{Stud66}), the case of polynomial regression on a compact interval on $\mathbb{R}$ 
has been widely studied.   In this frame, the theory is almost complete an many thing can be said about the optimal solutions for the  design problem (see \cite{Holg93}). Roughly speaking, the optimal design points are related to the zeros of orthogonal polynomials built on an equilibrum measure. We refer to the excelent book of Dette and Studden \cite{dette1997theory} and reference therein for a complete overview on the subject. In the one dimensional frame, other systems of functions $\Phi$ (trigonometric functions or some $T$-system, see~\cite{krein1977markov} for a definition) are studied in a same way in \cite{dette1997theory}, \cite{TaiStu85} and \cite{Imhof01}. In the multidimensional case, even for polynomial systems, very few case of explicit solutions are known. Using tensoring arguments the case of a rectangle is treated in \cite{dette1997theory}. Particular  models  of degree two are studied in \cite{dette2014optimal} and \cite{Prozac94}. Away from these particular cases, the construction of the optimal design relies on numerical optimization procedures. The case of the determinant ($D$-optimality) is studied  for example in \cite{Wyn70} and \cite{maxdet}. An other criterion based on matrix conditioning is developed in~\cite{marechal2014k}. General optimization algorithm are discussed in \cite{Fedo10} and \cite{Don07}. 
In  the frame of fixed given support points efficient SDP based algorithms are proposed and studied in \cite{Sag11} and \cite{Sag15}.  Let us mention, the paper \cite{maxdet} which is one of the original motivation to develop SDP solvers, especially for Max Det Problems (corresponding to $D$-optimal design) and the so-called problem of analytical centering.

\subsection{Contribution}
For the first time, this paper introduces a general method to compute the approximate $D$-optimal designs on a large variety of design spaces that we referred to as semi-algebraic sets, see~\cite{lasserre} for a definition. This family can be understood as any sets given by intersections and complements of the level sets of multivariate polynomials. The theoretical guarantees are given by Theorems~\ref{th-sdp} and \ref{th3-asymptotics}.  We apply the moment-sum-of-squares hierarchy (a.k.a. the Lasserre hierarchy) of SDP problems to solve numerically and approximately the optimal design problem
They show the convergence of our method towards the optimal information matrix as the order of the hierarchy increases. Furthermore, we show that our method recovers the optimal design when finite convergence of this hierarchy occurs. To recover the geometry of the design we use SDP duality theory and the Christoffel polynomial.
We have run several numerical experiments for which finite convergence holds leading to a surprisingly fast and reliable method to compute optimal designs. As illustrated by our examples, using Christoffel polynomials of degrees higher than two allows to reconstruct designs with points in the interior of the domain, contrasting with the classical use of ellipsoids for linear regressions.

\subsection{Outline of the paper}
In Section \ref{sec:notation}, after introducing necessary notation, we shortly explain some basics on moments and moment matrices, and present the approximation of the moment cone via the Lasserre hierarchy. Section \ref{sec:OptDesign} is dedicated to further describing optimum designs and their approximations. At the end of the section we propose a two step procedure to solve the approximate design problem. Solving the first step is subject to Section \ref{idealProblem}. There, we find a sequence of moments associated with the optimal design measure. Recovering this measure (step two of the procedure) is discussed in Section \ref{recoverMeasure}. We finish the paper with some illustrating examples and a conclusion.

\section{Polynomial optimal design and moments}
\label{sec:notation}

This section collects preliminary material on semialgebraic sets, moments and moment matrices, using the notations of \cite{lasserre}. This material will be used to restrict our attention to \textit{polynomial} optimal design problems with polynomial regression functions and semi-algebraic design spaces.

\subsection{Polynomial optimal design}

Denote by $\R[x]$ the vector space of real polynomials in the variables $x=(x_1,\dotsc,x_n)$, and for $d\in\N$, define $\R[x]_d := \{p \in \R[x] : \deg{p} \leq d\}$ where $\deg p$ denotes the total degree of $p$.

In this paper we assume that the regression functions are multivariate polynomials, i.e. $\Phi=(\phi_1,\ldots,\phi_p) \in (\R[x]_d)^p$. Moreover, we consider that the design space $\mathcal X\subset \R^n$ is a given closed basic semi-algebraic set
\eq
\label{eq:defDesignSpaceSemiAegebraic}
\mathcal X:= \{ x \in \R^m : g_j(x) \geq 0, \:j=1,\ldots,m\}
\qe
for given polynomials $g_j \in {\R}[x]$, $j=1,\ldots,m$, whose degrees are denoted by $d_j$, $j=1,\ldots,m$. Assume that ${\mathcal X}$ is compact with an algebraic certificate of compactness. For example, one of the polynomial inequalities $g_j(x)\geq 0$ should be of the form $R^2-\sum_{i=1}^n x_i^2 \geq 0$ for a sufficiently large constant $R$.

Notice that those assumptions cover a large class of problems in optimal design theory, see for instance \cite[Chapter 5]{dette1997theory}. In particular, observe that the design space $\mathcal X$ defined by~\eqref{eq:defDesignSpaceSemiAegebraic} is not necessarily convex and note that the polynomial regressors $ \Phi$ can handle incomplete $m$-way $d$th degree polynomial regression.

The monomials $x_1^{\alpha_1}\cdots x_n^{\alpha_n}$, with $\alpha =(\alpha_1,\dotsc,\alpha_n) \in \N^n$, form a basis of the vector space $\R[x]$. We use the multi-index notation $x^\alpha:=x_1^{\alpha_1}\cdots x_n^{\alpha_n}$ to denote these monomials. In the same way, for a given $d\in \N$ the vector space $\R[x]_d$ has dimension $\binom{n+d}{n}$ with basis $(x^\alpha)_{|\alpha|\leq d}$, where $|\alpha| := \alpha_1+\cdots+\alpha_n$. We write
\[
%\label{eq:momvec}
	\v_d(x):= (\underbrace{1}_{\text{degree 0}},\underbrace{x_1,\dotsc,x_n}_{\text{degree 1}},\underbrace{x_1^2,x_1x_2,\dotsc,x_1x_n,x_2^2,\dotsc,x_n^2}_{\text{degree 2}},\dotsc,\underbrace{x_1^d,\dotsc,x_n^d}_{\text{degree d}}\ \ )^T
\]
for the column vector of the monomials ordered according to their degree, and where monomials of the same degree are ordered with respect to the lexicographic ordering.

The cone $\Ms_+({\mathcal X})$ of nonnegative Borel measures supported on $\mathcal X$ is understood as the dual to the cone of nonnegative elements of the space $\Cs({\mathcal X})$ of continuous functions on $\mathcal X$.

\subsection{Moments, the moment cone and the moment matrix}
Given $\mu \in \Ms_+(\mathcal{X})$ and $\alpha \in \N^n$, we call
\[
	y_\alpha = \int_\mathcal{X} x^\alpha d\mu
\]
the moment of order $\alpha$ of $\mu$. Accordingly, we call the sequence $\y = (y_\alpha)_{\alpha\in\N^n}$ the moment sequence of~$\mu$. Conversely, we say that a sequence $\y = (y_\alpha)_{\alpha\in\N^n} \subseteq\R$ has a \textit{representing measure}, if there exists a measure~$\mu$ such that $\y$ is its moment sequence.

We denote by $\mathcal{M}_d(\mathcal{X})$ the convex cone of all truncated sequences $\y = (y_\alpha)_{|\alpha|\leq d}$ which have a representing measure supported on $\mathcal{X}$. We call it the \textit{moment cone} of $\mathcal{X}$. It can be expressed as
\eq
\label{eq:momcone}
	\mathcal{M}_d(\mathcal{X})
	\,:=\,\{\y\in\R^{\binom{n+d}{n}}: \exists\,\mu\in \Ms_+(\mathcal{X})\mbox{ s.t. }y_\alpha=\int_\mathcal{X} x^\alpha\,d\mu,\:\:\forall \alpha \in \N^n, \:|\alpha|\leq d\}.
\qe
We also denote by $\mathcal{P}_d(\mathcal{X})$ the convex cone of polynomials of degree at most $d$ that are  nonnegative on $\mathcal X$. When $\mathcal X$ is compact then 
$\mathcal{M}_d(\mathcal{X})=\mathcal{P}_d(\mathcal{X})^\star$
and $\mathcal{P}_d(\mathcal{X})=\mathcal{M}_d(\mathcal{X})^\star$ (see e.g. \cite{lowner}[Lemma 2.5]).

When the design space is given by the univariate interval $\mathcal X=[a,b]$, i.e., $n=1$, then this cone is representable using positive semidefinite Hankel matrices, which implies that convex optimization on this cone can be carried out with efficient interior point algorithms for \textit{semidefinite programming}, see e.g.~\cite{maxdet}. Unfortunately, in the general case, there is no efficient representation of this cone. It has actually been shown in \cite{scheiderer} that the moment cone is not \textit{semidefinite representable}, i.e. it cannot be expressed as the projection of a linear section of the cone of positive semidefinite matrices. However, we can use semidefinite approximations of this cone as discussed in Section \ref{sec:approxMomcon}.

Given a sequence $\y = (y_\alpha)_{\alpha\in\N^n} \subseteq\R$ we define the linear functional $L_\y:\R[x]\to\R$ which maps a polynomial $f=\sum_{\alpha\in\N^n} f_\alpha x^\alpha$ to
\[
	L_\y(f) = \sum_{\alpha \in \N^n} f_\alpha y_\alpha.
\]
A sequence $\y = (y_\alpha)_{\alpha\in\N^n}$ has a representing measure $\mu$ supported on $\mathcal{X}$ if and only if $L_\y(f)\geq 0$ for all polynomials $f\in \R[x]$ nonnegative on $\mathcal{X}$ \cite[Theorem 3.1]{lasserre}.

The \textit{moment matrix} of a truncated sequence $\y = (y_\alpha)_{|\alpha|\leq 2d}$ is the $\binom{n+d}{n}\times\binom{n+d}{n}$-matrix $M_d(\y)$ with rows and columns respectively indexed by $\alpha \in \N^n,  |\alpha|,|\beta| \leq d$ and whose entries are given by
\[
	M_d(\y)(\alpha,\beta) = L_\y(x^\alpha x^\beta) = y_{\alpha+\beta}.
\]
It is symmetric and linear in $\y$, and if $\y$ has a representing measure, then $M_d(\y)$ is \textit{positive semidefinite}, denoted by $M_d(\y) \succcurlyeq 0$.

Similarly, we define the \textit{localizing matrix} of a polynomial $f=\sum_{|\alpha|\leq r} f_\alpha x^\alpha \in \R[x]_r$ of degree $r$ and a sequence $\y = (y_\alpha)_{|\alpha|\leq 2d+r}$ as the $\binom{n+d}{n}\times\binom{n+d}{n}$-matrix $M_d(f\y)$ with rows and columns respectively indexed by $\alpha \in \N^n,  |\alpha|,|\beta| \leq d$ and whose entries are given by
\[
	M_d(f\y)(\alpha,\beta) = L_\y(f(x)\,x^\alpha x^\beta) = \sum_{\gamma\in\N^n} f_\gamma y_{\gamma+\alpha+\beta}.
\]
If $\y$ has a representing measure $\mu$, then $M_d(f\y) \succcurlyeq 0$ for $f\in\R[x]_d$ whenever the support of $\mu$ is contained in the set $\{x\in\R^n: f(x)\geq 0\}$.

Since $\mathcal{X}$ is basic semialgebraic with a certificate of compactness, by Putinar's theorem \cite[Theorem 3.8]{lasserre}, we also know the converse statement in the infinite case, namely $\y = (y_\alpha)_{\alpha\in\N^n}$ has a representing measure $\mu \in\Ms_+(\mathcal{X})$ if and only if for all $d\in\N$ the matrices $M_d(\y)$ and $M_d(g_j\y), \ j=1,\dots,m$, are positive semidefinite.

\subsection{Approximations of the moment cone}\label{sec:approxMomcon}
Letting $v_j:=\lceil d_j/2\rceil$, $j=1,\ldots,m$, for half the degree of the $g_j$, by Putinar's theorem, we can approximate the moment cone $\mathcal{M}_{2d}(\mathcal{X})$ by the following semidefinite representable cones for $\delta\in\N$
\begin{multline*}
	\mathcal{M}_{2(d+\delta)}^{\mathsf{SDP}}(\mathcal{X}):= \{\y\in\R^{\binom{n+2d}{n}} : \exists \y_\delta\in\R^{\binom{n+2(d+\delta)}{n}}\mbox{ such that } (y_{\delta,\alpha})_{|\alpha|\leq 2d}=\y \mbox{ and}\\
	M_{d+\delta}(\y_\delta)\succcurlyeq 0,\ M_{d+\delta-v_j}(g_j\y_\delta)\succcurlyeq 0,\ j=1,\dotsc,m\}.
\end{multline*}
By semidefinite representable we mean that the cones are projections of linear sections of semidefinite cones. Since $\mathcal{M}_{2d}(\mathcal{X})$ is contained in every $\mathcal{M}_{2(d+\delta)}^{\mathsf{SDP}}(\mathcal{X})$, they are outer approximations of the moment cone. Moreover, they form a nested sequence, so we can build the hierarchy
\eq
\label{eq:hierarchy}
	\mathcal{M}_{2d}(\mathcal{X})\subseteq\dots\subseteq\mathcal{M}_{2d+2}^{\mathsf{SDP}}(\mathcal{X})\subseteq\mathcal{M}_{2d+1}^{\mathsf{SDP}}(\mathcal{X})\subseteq\mathcal{M}_{2d}^{\mathsf{SDP}}(\mathcal{X}).
\qe
This hierarchy actually converges, meaning $\mathcal{M}_{2d}(\mathcal{X})=\overline{\bigcap_{\delta=0}^\infty \mathcal{M}_{2d+\delta}^{\mathsf{SDP}}(\mathcal{X})}$, where $\overline{A}$ denotes the topological closure of the set $A$.

\section{Approximate Optimal Design}
\label{sec:OptDesign}

\subsection{Problem reformulation in the multivariate polynomial case}
For all $i=1,\ldots,p$ and $x\in {\mathcal X}$, let $\varphi_{i}(x):=\sum_{|\alpha|\leq d}a_{i,\alpha}x^{\alpha}$ with appropriate $a_{i,\alpha}\in\R$. Define for $\mu\in\Ms_+(\mathcal{X})$ with moment sequence $\y$ the information matrix
\[
M(\mu)=\Big(\int_{{\mathcal X}}\varphi_{i}\varphi_{j}\mathrm d\mu\Big)_{1\leq i,j\leq p}=\Big(\sum_{|\alpha|,|\beta|\leq d}a_{i,\alpha}a_{j,\beta}y_{\alpha+\beta}\Big)_{1\leq i,j\leq p}=\sum_{|\gamma|\leq 2d} A_{\gamma}y_{\gamma}
\]
where we have set $A_{\gamma}:=\Big(\sum_{\alpha+\beta=\gamma}a_{i,\alpha}a_{j,\beta}\Big)_{1\leq i,j\leq p}$ for $|\gamma|\leq 2d$.

Further, let $\mu=\sum_{i=1}^\ell w_i\delta_{x_i}$ where $\delta_x$ denotes the Dirac measure at the point $x\in\mathcal X$ and observe that $M(\mu)=\sum_{i=1}^{\ell}w_{i}\Phi(x_{i})\Phi^{\top}(x_{i})$ as in \eqref{eq:defInformationMatrix}.

The optimization problem
\begin{align}
\label{eq:Optimumdesigns_NLP}
	& \max\log\det M\\
	& \text{s.t. } M=\sum_{|\gamma|\leq 2d}A_{\gamma}y_{\gamma} \succcurlyeq 0,\quad y_{\gamma} = \sum_{i=1}^\ell \frac{n_i}N{x^{\gamma}_i}, \quad \sum_{i=1}^{\ell} n_i = N,\notag\\
& \quad\quad x_i \in  {\mathcal X}, \:n_i \in {\mathbb N}, \:i=1,\ldots,\ell\notag
\end{align}
where the maximization is with respect to $x_i$ and $n_i$, $i=1,\ldots,\ell$, subject to the constraint that the information matrix $M$ is positive semidefinite, is by construction equivalent to the original design problem~\eqref{eq:defOptimumdesigns}. In this form, problem (\ref{eq:Optimumdesigns_NLP}) is difficult because of the integrality constraints on the $n_i$ and the nonlinear relation between $\y$, $x_i$ and $n_i$. We will address these difficulties in the sequel by first relaxing the integrality constraints.

\subsection{Relaxing the integrality constraints}
In problem \ref{eq:Optimumdesigns_NLP}, the set of admissible frequencies $w_i=n_i/N$ is discrete, which makes it a potentially difficult combinatorial optimization problem. A popular solution is then to consider ``\textit{approximate}'' designs defined by 
\eq
\label{eq:defDesignApproximate}
\zeta:=\left(\begin{array}{ccc}x_{1} & \cdots & x_{\ell} \\ w_{1} & \cdots & w_{\ell}\end{array}\right)\,,
\qe
where the frequencies $w_{i}$ belong to the unit simplex ${\mathcal W}:=\{w \in {\mathbb R}^l : 0 \leq w_i \leq 1, \: \sum_{i=1}^\ell w_i=1\}$. Accordingly, any solution to~\eqref{eq:defOptimumdesigns} where the maximum is taken over all matrices of type~\eqref{eq:defDesignApproximate} is called ``\textit{approximate optimal design}'', yielding the following relaxation of problem \ref{eq:Optimumdesigns_NLP}
\begin{align}
\label{eq:ApproxOptimumDesigns_NLP}
	& \max\log\det M\\
	& \text{s.t. } M=\sum_{|\gamma|\leq 2d}A_{\gamma}y_{\gamma}\succcurlyeq 0,\quad y_{\gamma} = \sum_{i=1}^\ell w_i {x^{\gamma}_i},\notag\\
& \quad x_i \in  {\mathcal X}, \:w \in {\mathcal W}\notag
\end{align}
where the maximization is with respect to $x_i$ and $w_i$, $i=1,\ldots,\ell$, subject to the constraint that the information matrix $M$ is positive semidefinite. In this problem, the nonlinear relation between $\y$, $x_i$ and $w_i$ is still an issue.

\subsection{Moment formulation}
Let us introduce a two-step-procedure to solve the approximate optimal design problem~\eqref{eq:ApproxOptimumDesigns_NLP}. For this we will first again reformulate our problem.

Observe that by Carath\'eodory's theorem, the truncated moment cone $\mathcal M_{2d}(\mathcal{X})$ defined in \eqref{eq:momcone} is exactly
\[
\mathcal M_{2d}(\mathcal{X})=\{\y\in\R^{\binom{n+2d}{n}}\ :\  y_{\alpha} = \int_{\mathcal X} x^{\alpha}d\mu, \quad \mu=\sum_{i=1}^\ell w_i\delta_{x_i},\:x_i \in {\mathcal X},\: w\in {\mathcal W}\}\,
\]
so that problem \eqref{eq:defDesignApproximate} is equivalent to
\begin{align}
\label{eq:Step1_Moments}
	&\max\log\det M\\
	& \text{s.t. } M=\sum_{|\gamma|\leq 2d} A_{\gamma} y_{\gamma}\succcurlyeq 0,\notag\\
		&\qquad \y\in\mathcal M_{2d}(\mathcal{X}),\ y_0=1\notag
\end{align}
where the maximization is now with respect to the sequence $\y$. Moment problem \eqref{eq:Step1_Moments} is finite-dimensional and convex, yet the constraint $\y\in\mathcal M_{2d}(\mathcal{X})$ is difficult to handle. We will show that by approximating the truncated moment cone $\mathcal{M}_{2d}(\mathcal{X})$ by a nested sequence of semidefinite representable cones as indicated in \eqref{eq:hierarchy}, we obtain a hierarchy of finite dimensional semidefinite programming problems converging to the optimal solution of \eqref{eq:Step1_Moments}. Since semidefinite programming problems can be solved efficiently, we can compute a numerical solution to problem \eqref{eq:defDesignApproximate}.

This describes step one of our procedure. The result of it is a sequence $\y^\star$ of moments. Consequently, in a second step, we need to find a representing atomic measure $\mu^\star$ of $\y^\star$ in order to identify the approximate optimum design $\zeta^\star$. 

\begin{algorithm}[t]
  \SetAlgoLined
  \KwData{A design space $\mathcal X$ defined by polynomials $g_j, j=1,\ldots,m,$ as in \eqref{eq:defDesignSpaceSemiAegebraic}, and  polynomial regressors $\Phi=(\varphi_{1},\ldots,\varphi_{p})$.}
  \KwResult{An approximate $\phi_{q}$-optimal design $\zeta^\star=\left(\begin{array}{ccc}x_{1}^\star & \cdots & x_{\ell}^\star \\ w_{1}^\star & \cdots & w^\star_{\ell}\end{array}\right)$ for $q=-\infty,-1,0$ or $1$.}
    \BlankLine
    \begin{enumerate}
    \item Find the truncated moments $(y^\star_{\alpha})_{|\alpha|\leq 2d}$ that maximizes \eqref{eq:Step1_Moments}.
 \item Find a representing measure $\mu^\star=\sum_{i=1}^\ell w_i^\star\delta_{x_i^\star}\geq0$ of $(y^\star_{\alpha})_{|\alpha|\leq 2d}$, for $x_i^\star \in {\mathcal X}$ and $w_i^\star \in {\mathcal W}$.
\end{enumerate}
  \caption{Approximate optimal design.}
\label{alg:GeneralAlgApproxDesign}
\end{algorithm}

\section{The ideal problem on moments and its approximation}
\label{idealProblem}

For notational simplicity, let us use the standard monomial basis of $\R[x]_d$ for the regression functions, meaning $\Phi = (\varphi_1,\dotsc,\varphi_p):=(x^\alpha)_{|\alpha|\leq d}$ with $p=\binom{n+d}{n}$. Note that this is not a restriction, since one can get the results for other choices of $\Phi$ by simply performing a change of basis. Different polynomial bases can be considered and one may consult the standard framework described by \cite[Chapter 5.8]{dette1997theory}. For the sake of conciseness, we do not expose the notion of incomplete $q$-way $m$-th degree polynomial regression here but the reader may remark that the strategy developed in this paper can handle such a framework. 

\subsection{Christoffel polynomials}
It turns out that the (unique) optimal solution $\y\in\mathcal{M}_{2d}(\mathcal{X})$ of (\ref{eq:Step1_Moments}) can be characterized 
in terms of the {\it Christoffel polynomial} of degree $2d$ associated with
an optimal measure $\mu$ whose moments up to order $2d$ coincide with $\y$.

\begin{defn}\label{def:christoffel}
Let $\y\in\R^{\binom{n+2d}{n}}$ be such that $\M_d(\y)\succ0$. Then there exists a family of orthonormal polynomials $(P_\alpha)_{\vert\alpha\vert\leq d}\subseteq\R[\x]_d$ satisfying
\[
L_\y(P_\alpha\,P_\beta)\,=\,\delta_{\alpha=\beta}\quad \mbox{and}\quad L_\y(\x^\alpha\,P_\beta)\,=\,0\quad\forall\alpha\prec\beta,
\]
where monomials are ordered with the lexicographical ordering on $\N^n$. We call the polynomial
\[
p_d:\ \x\mapsto p_d(x)\,:=\,\sum_{|\alpha|\leq d}P_\alpha(\x)^2,\quad x\in\R^n,
\]
the \textit{Christoffel polynomial} (of degree $d$) associated with $\y$.
\end{defn}

The Christoffel polynomial\footnote{In fact what is referred to in the literature is
its reciprocal $x\mapsto1/p_d(x)$ called the {\it Chistoffel function}.} can be expressed in different ways. For instance via the inverse of the moment matrix
by%{\bf Didier :  Jean, ins\`ere STP une r\'ef\'erence aux articles avec Edouard}
\[
p_d(x) = \v_d(\x)^T\M_d(\y)^{-1}\v_d(\x),\quad\forall x\in\R^n,
\]
or via its extremal property
\[\frac{1}{p_d(\xi)}\,=\,\min_{P\in\R[\x]_{d}}\Big\{\int P(x)^2\,d\mu(x)\::\: P(\xi)=1\,\Big\},\qquad \forall\xi\in\R^n,\]
when $\y$ has a representing measure $\mu$. (When $\y$ does not have a representing measure
$\mu$ just replace $\int P(x)^2d\mu(x)$ with $L_\y(P^2)\,(=P^T\M_d(\y)\,P$)). For more details
the interested reader is referred to \cite{nips} and the references therein.

\subsection{The ideal problem on moments}
The ideal formulation \eqref{eq:Step1_Moments} of our approximate optimal design problem reads
\begin{equation}
\label{sdp-ideal}
\begin{array}{rl}
	\rho=\ \displaystyle\max_{\y} &\log \det \M_d(\y)\\
		 \text{s.t.} &\y\in\mathcal M_{2d}(\mathcal{X}),\ y_0=1.
%\rho=\displaystyle\sup_\y \:\{\:\log {\rm det}(\M_d(\y) :\: y_0=1; \quad \y\in M_d(\K)\,\}
\end{array}\end{equation}
%where $M_d(\K)$ is the convex cone
%\[
%M_d(\K)\,:=\,\{\y\in\R^{\binom{n+2d}{n}}: \exists \mu\in M(\K)\mbox{ s.t. }y_\alpha=\int_\K \x^\alpha\,d\mu,\quad\forall\alpha\in\N^n_{2d}\}.
%\]
%If $\K$ is compact the dual cone $M_d(\K)^\star$ is the convex cone
%$\mathcal{P}_d(\K)\subset\R[\x]_{2d}$ of polynomials of degree at most $2d$ that are nonnegative on $\K$.

For this we have the following result

\begin{thm}
\label{th-ideal}
Let $\K\subseteq\R^n$ be compact with nonempty interior.
Problem (\ref{sdp-ideal}) is a convex optimization problem with a unique optimal solution 
$\y^\star\in \mathcal{M}_{2d}(\K)$. It is the vector of moments (up to order $2d$) of a measure $\mu^\star$ supported on at least $\binom{n+d}{n}$ and at most $\binom{n+2d}{n}$ points in the set $\Omega:=\{\x\in\K: \binom{n+d}{n}-p_d^\star(\x)=0\}$ where $\binom{n+d}{n}-p_d^\star\in\mathcal{P}_{2d}(\K)$ and $p^\star_d$ is the Christoffel polynomial
\begin{equation}
\label{christoffel}
\x\mapsto p_d^\star(\x):=\v_d(\x)^T\M_d(\y^\star)^{-1}\v_d(\x).\end{equation}
\end{thm}
\begin{proof}
First, let us prove that \eqref{sdp-ideal} has an optimal solution. The feasible set is nonempty 
(take as feasible point the vector $\y\in \mathcal{M}_{2d}(\K)$ associated
with the Lebesgue  measure on $\K$, scaled to be a probability measure) with finite associated
objective value, because $\det \M_d(\y)>0$. Hence, $\rho>-\infty$ and in addition
Slater's condition holds because $\y\in{\rm int} \mathcal{M}(\mathcal{X})$ (that is,
$\y$ is a strictly feasible solution to (\ref{sdp-ideal})).

Next, as $\K$ is compact there exists $M>1$ such that $\displaystyle \int_\K x_i^{2d}\,d\mu< M$ 
for every probability measure $\mu$ on~$\K$ and every $i=1,\ldots,n$. Hence, 
$\max\{y_0,\ \max_i\{L_y(x_i^{2d})\}\}<M$ which by \cite{lass-netzer} implies that
$\vert y_\alpha\vert\leq M$ for every $|\alpha|\leq 2d$, which in turn implies
that the feasible set of (\ref{sdp-ideal}) is compact. As the objective function is continuous
and $\rho>-\infty$, problem (\ref{sdp-ideal}) has an optimal solution $\y^\star\in \mathcal{M}_{2d}(\K)$.

Furthermore, an optimal solution $\y^\star\in \mathcal{M}_{2d}(\K)$ is unique because the objective function is 
strictly convex and the feasible set is convex. In addition, since $\rho>-\infty$, 
$\det \M_d(\y^\star)\neq0$ and so $\M_d(\y^\star)$ is non singular.

Next, writing $\B_\alpha$, $\alpha\in\N^n_{2d}$, for the real matrices satisfying $\sum_{|\alpha|\leq 2d}\B_\alpha x^\alpha = \v_d(\x)\v_d(\x)^T$ and $\langle \mathbf{A},\B\rangle= {\rm trace} (\mathbf{A}\B)$ for two matrices $\mathbf{A}$ and $\B$, the necessary {\it Karush-Kuhn-Tucker} optimality conditions\footnote{For the optimization problem $\min\,\{f(x): Ax=b;x\in C\}$ where $f$ is differentiable, $A\in\R^{m\times n}$ and $C\subset\R^n$ is a nonempty closed convex cone, the KKT-optimality conditions at a feasible point $x$ state that there exists $\lambda^\star\in \R^m$ and 
$u\in C^\star$ such that $\nabla f(x)-A^T\lambda^\star=u^\star$ and $\langle x,u^\star\rangle=0$. Slater's condition holds if there exists a feasible solution
$x\in{\rm int}(C)$, in which case the KKT-optimality conditions are necessary and sufficient if $f$ is convex.}
(in short KKT-optimality conditions) state that an optimal solution $\y^\star\in\mathcal{M}_{2d}(\mathcal{X})$ should satisfy
\[\lambda^\star\,e_0-\nabla \log \det \M_d(\y^\star)\,=\,p^\star\in\mathcal{M}_{2d}(\mathcal{X})^\star\:(=\mathcal{P}_{2d}(\mathcal{X})),\]
(where $e_0=(1,0,\ldots0)$ and $\lambda^\star$ is the dual variable associated with the constraint
$y^\star=1$), with the complementarity condition $\langle \y^\star,p^\star\rangle=0$.
That is:
\begin{equation}
\label{a1-ideal}
\left(1_{\alpha=0}\,\lambda^\star_0-\langle\M_{d}(\y^\star)^{-1},\B_\alpha\rangle\right)_{|\alpha|\leq 2d}\,=\,p^\star\,\in\,\mathcal{P}_{2d}(\K);\quad \langle \y^\star,p^\star\rangle =0.\end{equation}
Multiplying (\ref{a1-ideal}) term-wise by $y^\star_\alpha$, summing up and invoking the complementarity condition, yields
\[\lambda^\star_0\,=\,\lambda^\star_0\,\y^\star_0\,=\,\langle\M_{d}(\y^\star)^{-1},\sum_{|\alpha|\leq 2d}y^\star_\alpha\B_\alpha\rangle\,=\,
\langle\M_{d}(\y^\star)^{-1},\M_d(\y^\star)\rangle
\,=\,\binom{n+d}{n}.\]
Similarly, multiplying (\ref{a1-ideal}) term-wise by $\x^\alpha$ and summing up yields
\begin{eqnarray}
\x\mapsto p^\star(\x)
&:=&\binom{n+d}{n}-\langle\M_{d}(\y^\star)^{-1},\sum_{|\alpha|\leq 2d}\B_\alpha\x^\alpha\rangle\notag\\
&=&\binom{n+d}{n}-\langle\M_{d}(\y^\star)^{-1},\v_d(\x)\v_d(\x)^T\rangle\notag\\
&=&\binom{n+d}{n}-\sum_{|\alpha|\leq d}P_\alpha(\x)^2\geq0\quad\forall \x\in\K,\label{eq:kkt_pstar}
\end{eqnarray}
where the $(P_\alpha)$, $\alpha\in\N^n_d$, are the orthonormal polynomials (up to degree $d$) w.r.t. $\mu^\star$, where $\mu^\star$ is a representing measure of $\y^\star$ (recall that $\y^\star\in\mathcal{M}_{2d}(\mathcal{X})$). 
Therefore $p^\star=\binom{n+d}{n}-p^\star_d\in\mathcal{P}_{2d}(\mathcal{X})$ where~$p^\star_d$ is the Christoffel polynomial of degree $2d$ associated with $\mu^\star$.
Next, the complementarity condition $\langle\y^\star,p^\star\rangle=0$ reads
\[
\int_\K \underbrace{p^\star(\x)}_{\geq0\mbox{ on }\K}\,d\mu^\star(\x) = 0
\]
which implies that the support of $\mu^\star$ is included in the algebraic set $\{\x\in\K: p^\star(\x)=0\}$.
Finally, that $\mu^\star$ is an atomic measure supported on at most $\binom{n+2d}{n}$ points follows from Tchakaloff's theorem \cite[Theorem B.12]{lasserre} which states that for every finite Borel probability measure on $\K$ and every $t\in\N$, there exists an atomic measure $\mu_t$ supported on $\ell\leq\binom{n+t}{n}$ points such that all moments of $\mu_t$ and $\mu^\star$ agree up to order $t$. So let $t=2d$. Then $\ell\leq \binom{n+2d}{n}$. If $\ell<\binom{n+d}{n}$, then ${\rm rank} \M_d(\y^\star) <\binom{n+d}{n}$ in contradiction to $\M_d(\y^\star)$ being non-singular. Therefore, $\binom{n+d}{n}\leq \ell\leq \binom{n+2d}{n}$.
\end{proof}
So we obtain a nice characterization of the unique optimal solution $\y^\star$ of (\ref{sdp-ideal}). It is the vector of moments up to order $2d$ of a measure $\mu^\star$ supported on finitely many (at least ${n+d\choose n}$ and at most ${n+2d\choose n}$) points of $\mathcal{X}$. This support of $\mu^\star$ consists of
zeros of the equation $\binom{n+d}{n}-p^\star_d(x)=0$, where $p^\star_d$ is the Christoffel polynomial
associated with $\mu^\star$. Moreover the level set $\{x: p^\star_d(x)\leq {n+d\choose n}\}$ contains $\mathcal{X}$
and intersects $\mathcal{X}$ precisely at the support of $\mu^\star$.

\subsection{The SDP approximation scheme}
Let $\K\subseteq\R^n$ be as defined in \eqref{eq:defDesignSpaceSemiAegebraic}, assumed to be compact.
So with no loss of generality (and possibly after scaling), assume that $\x\mapsto g_1(\x)=1-\Vert\x\Vert^2\geq 0$ is one of the constraints defining $\mathcal{X}$.

Since the ideal moment problem (\ref{sdp-ideal}) involves the moment cone  $\mathcal{M}_{2d}(\mathcal{X})$ which is not SDP representable, we use the hierarchy \eqref{eq:hierarchy} of outer approximations of the moment cone to relax problem (\ref{sdp-ideal}) to an SDP problem. So for a fixed integer $\delta\geq1$ we consider the problem
\begin{equation}
\label{sdp}
\begin{array}{rl}
	\rho_\delta=\ \displaystyle\max_{\y} &\log \det \M_d(\y)\\
		 \text{s.t.} &\y\in\mathcal M_{2(d+\delta)}^{\mathsf{SDP}}(\mathcal{X}),\ y_0=1.
%\rho=\displaystyle\sup_\y \:\{\:\log {\rm det} \M_d(\y) :\: y_0=1; \quad \y\in M_d(\K)\,\}
\end{array}\end{equation}
Since (\ref{sdp}) is a relaxation of the ideal problem \eqref{sdp-ideal}, then necessarily $\rho_\delta\geq\rho$ for all $\delta$.
In analogy with Theorem \ref{th-ideal} we have the following result

\begin{thm}
\label{th-sdp}
Let $\K\subseteq\R^n$ as in \eqref{eq:defDesignSpaceSemiAegebraic} be compact and with nonempty interior. Then
\begin{enumerate}[label={\alph*)}]
\item SDP problem \eqref{sdp} has a unique optimal solution $\y^\star_d\in \R^{\binom{n+2d}{n}}$. 
\item The moment matrix $\M_d(\y^\star_d)$ is positive definite. Let $p^\star_d$ be the Christoffel polynomial associated with $\y^\star_d$. Then
$\binom{n+d}{n}-p^\star_d(\x)\geq 0$ for all $\x\in \K$ and $L_{\y^\star_d}(\binom{n+d}{n}-p^\star_d)=0$.
\end{enumerate}
\end{thm}
\begin{proof}\begin{enumerate}[label={\alph*)}]
\item Let us prove that (\ref{sdp}) has an optimal solution. The feasible set is nonempty, since we can take as feasible point the vector $\tilde{\y}$ associated with the Lebesgue  measure on $\K$, scaled to be a probability measure. Because $\det \M_d(\tilde{\y})>0$, the associated objective value is finite. Hence, Slater's condition holds for (\ref{sdp}) and $\rho_\delta>-\infty$.
\par\medskip\noindent
Next, let $\y$ be an arbitrary feasible solution and $\y_\delta\in\R^{\binom{2(d+\delta)+n}{n}}$ an arbitrary lifting 
of $\y$ (recall the definition of $\mathcal M_{2(d+\delta)}^{\mathsf{SDP}}(\mathcal{X})$). As $g_1(\x)=1-\Vert\x\Vert^2$ and $\M_{d+\delta-1}(g_1\,\y_\delta)\succcurlyeq0$, we have
\[
L_{\y_\delta}(x_i^{2(d+\delta)})\leq 1,\quad i=1,\ldots,n,
\]
and so by \cite{lass-netzer}
\begin{equation}
\label{bound}
\vert y_{\delta,\alpha}\vert \,\leq\,\max\{\underbrace{y_{\delta,0}}_{=1},\ \max_i\{L_{\y_\delta}(x_i^{2(d+\delta)})\}\}\,\leq\,1\qquad\forall|\alpha|\leq 2(d+\delta).
\end{equation}
This implies that the set of feasible liftings $\y_\delta$ is compact which implies that there is an optimal $\y^\star_\delta\in\R^{\binom{2(d+\delta)+n}{n}}$. As a consequence, the subvector $\y^\star_d=(\y^\star_{\delta,\alpha})_{|\alpha|\leq 2d}\in\R^{\binom{2d+n}{n}}$ is an optimal solution to \eqref{sdp}. It is unique due to strict convexity of the objective function.\par\medskip

\item As $\rho_{\delta}>-\infty$, we have $\det \M_d(\y^\star_d)>0$. Now, write $\langle \mathbf{A},\B\rangle= {\rm trace} (\mathbf{A}\B)$ for two matrices $\mathbf{A}$ and $\B$ and let $\B_\alpha,\tilde{\B}_\alpha$ and $\C_{j\alpha}$ be real symmetric matrices such that
\begin{eqnarray*}
\sum_{|\alpha|\leq 2d} \B_\alpha \x^\alpha&=&\v_d(\x)\,\v_d(\x)^T\\
\sum_{|\alpha|\leq 2(d+\delta)} \tilde{\B}_\alpha \x^\alpha&=&\v(\x)_{d+\delta}\,\v_{d+\delta}(\x)^T\\
\sum_{|\alpha|\leq 2(d+\delta)} \C_{j\alpha} \x^\alpha&=&g_j(\x)\,\v_{d+\delta-v_j}(\x)\,\v_{d+\delta-v_j}(\x)^T,\quad j=1,\ldots,m.
\end{eqnarray*}
Problem \eqref{sdp} is a convex optimization problem which can be rewritten as
\[\max_{\y\in\R^{s(2d+2\delta)}}\,\{\log{\rm det} \M_d(\y):\: \M_{d+\delta}(\y)\succcurlyeq 0,\M_{d+\delta-v_j}(g_j\y)\succcurlyeq 0,\ j=1,\dotsc,m\}\]
for which Slater's condition holds. For all its optimal solutions
$\y^\star_\delta=(y^\star_{\delta,\alpha})\in\R^{\binom{2(d+\delta)+n}{n}}$, the restriction $\y^\star_d=(y^\star_{d,\alpha})=(y^\star_{\delta,\alpha})$, $\alpha\in\N^n_{2d}$, is the unique
optimal solution of (\ref{sdp}).  
Hence  at an optimal solution 
$\y^\star$, the necessary KKT-optimality conditions state that
\begin{equation}
\label{a1}
1_{\alpha=0}\,\lambda^\star-\langle\M_{d}(\y^\star_d)^{-1},\B_\alpha\rangle\,=\,\langle \X_0,\tilde{\B}_\alpha\rangle+\sum_{j=1}^m\langle \X_j,\C^j_\alpha\rangle,\quad\forall \alpha\in\N^n_{2d+2\delta},\end{equation}
for some {\it ``dual variables"} $\lambda^\star\in\R$, $\X_j\succcurlyeq0$, $j=0,\ldots,m$. 
We also have the complementarity conditions
\begin{equation}
\label{a3}
\langle \M_{d+\delta}(\y^\star_\delta),\X_0\rangle =0;\quad \langle \M_{d+\delta-v_j}(\y^\star_\delta\,g_j),\X_j\rangle =0,\quad j=1,\ldots,m.
\end{equation}
Multiplying by $y^\star_{\delta,\alpha}$, summing up and using the complementarity conditions \eqref{a3} yields
\begin{equation}
\label{a4}
\lambda^\star-\underbrace{\langle \M_d(\y^\star_d)^{-1},\M_d(\y^\star_d)\rangle}_{=\binom{n+d}{n}}\,=\,\underbrace{\langle \X_0,\M_{d+\delta}(\y^\star_\delta)\rangle}_{=0}
+\sum_{j=1}^m \underbrace{\langle \X_j,\M_{d+\delta-v_j}(g_j\,\y^\star_\delta)\rangle}_{=0},
\end{equation}
and so $\lambda^\star=\binom{n+d}{n}$.\par\medskip\noindent
On the other hand, multiplying by $\x^\alpha$ and summing up yields
\begin{eqnarray}
\nonumber
\lambda^\star-\v_d(\x)^T\M_d(\y^\star_d)^{-1}\v_d(\x)&=&
\langle \X_0,\sum_{|\alpha|\leq 2(d+\delta)}\tilde{\B}_\alpha\,x^\alpha\rangle 
+\sum_{j=1}^m\langle \X_j,\sum_{|\alpha|\leq 2(d+\delta-v_j)}\C^j_\alpha\,x^\alpha\rangle \\
\nonumber&=&\underbrace{\langle \X_0,\v(\x)_{d+\delta}\,\v_{d+\delta}(\x)^T\rangle}_{\sigma_0(x)} 
+\sum_{j=1}^mg_j(x)\,\underbrace{\langle \X_j,\v_{d+\delta-v_j}(\x)\,\v_{d+\delta-v_j}(\x)^T\rangle}_{\sigma_j(x)} \\
\label{a2}
&=&\sigma_0(\x)+\sum_{j=1}^n\sigma_j(\x)\,g_j(\x).
\end{eqnarray}
for some SOS polynomials $(\sigma_j)\subset\R[\x]$, $j=0,\ldots,m$. 
Let $p^\star_d\in\R[\x]_{2d}$ be the Christoffel polynomial associated with $\y^\star_d$. 
Since $\lambda^\star=\binom{n+d}{n}$, \eqref{a2} reads
\begin{equation}
\label{certificate}
\binom{n+d}{n}-p^\star_d(\x)\,=\,\sigma_0(\x)+\sum_{j=1}^n\sigma_j(\x)\,g_j(\x)\,\geq\,0\qquad\forall \x\in\K,
\end{equation}
and \eqref{a4} implies $L_{\y^\star_d}(p^\star_d)=0$.
\end{enumerate}
\end{proof}
\vspace{0.3cm}

Hence, if the optimal  solution $\y^\star_d$ of (\ref{sdp}) is coming from a measure $\mu$ on $\K$, that is $\y^\star_d\in \mathcal{M}_{2d}(\K)$, then $\rho_\delta=\rho$ and $\y^\star_d$ is the unique optimal solution of  \eqref{sdp-ideal}. In addition, by Theorem \ref{th-ideal}, $\mu$ can be chosen to be atomic and supported on at least $\binom{n+d}{n}$ and at most $\binom{n+2d}{n}$ ``contact points" on the set $\{\x\in\K:\binom{n+d}{n}-p^\star_d(\x)=0\}$.

\subsection{Asymptotics}
To analyze what happens when $\delta$ tends to infinity, we denote the optimal solution $\y_d^\star\in\mathcal{M}^{\mathsf{SDP}}_{2(d+\delta)}\subseteq\R^{\binom{n+2d}{n}}$ of \eqref{sdp} by $\y_{d,\delta}^\star$ to indicate that it is the subvector $\y_{d,\delta}^\star=(y_{\delta,\alpha}^\star)_{|\alpha|\leq 2d}$ of a lifting $\y_\delta^\star\in\R^{s(2(d+\delta)}$. Now, we examine the behavior of $(\y^\star_{d,\delta})_{\delta\in\N}$ as $\delta\to\infty$.
\begin{thm}
\label{th3-asymptotics}
For every $\delta=0,1,2,\ldots,$ let $\y^\star_{d,\delta}$ be an optimal solution to \eqref{sdp} and $p^\star_{d,\delta}\in\R[\x]_{2d}$ the Christoffel polynomial associated with $\y^\star_d$ in Theorem \ref{th-sdp}.
Then
\begin{enumerate}[label={\alph*)}]
\item $\rho_\delta\to\rho$ as $\delta\to\infty$, where $\rho$ is the supremum in \eqref{sdp-ideal}.
\item For every $\alpha\in\N^n$ with $|\alpha|\leq 2d$
\begin{equation*}
\label{th3-1}
\lim_{\delta\to\infty}y^\star_{\delta,\alpha}\,=\,y^\star_\alpha\,=\,\int_\K \x^\alpha\,d\mu^\star,
\end{equation*}
where $\y^\star=(y^\star_\alpha)_{|\alpha|\leq 2d}\in \mathcal{M}_{2d}(\K)$ is the unique optimal solution to \eqref{sdp-ideal}.
\item $p^\star_{d,\delta}\to p^\star_d$ as $\delta\to\infty$, where $p^\star_d$ is the Christoffel polynomial associated with $\y^\star$ defined in \eqref{christoffel}.
\item If the dual polynomial $p^\star$ $($given by~\eqref{eq:kkt_pstar}$)$ can be represented as a Sum-Of-Squares $($namely, it satisfies~\eqref{a2}$)$ then $\y^\star_{d,\delta}$ is the unique optimal solution to \eqref{sdp-ideal} and $\y^\star_{d,\delta}$ has a representing measure $($namely the target measure $\zeta^\star)$.
\end{enumerate}
\end{thm}
\begin{proof}\begin{enumerate}[label={\alph*)}]
\item For every $\delta$ complete the lifted finite sequence $\y^\star_\delta\in\R^{\binom{n+2(d+\delta)}{n}}$ with zeros to make it an infinite sequence $\y^\star_\delta=(y^\star_{\delta,\alpha})_{\alpha\in\N^n}$. Therefore, every such $\y^\star_\delta$
can be identified with an element of $\ell_\infty$, the Banach 
space of finite bounded sequences equipped with the supremum norm.
Moreover, \eqref{bound} holds for every $\y^\star_\delta$. Thus, denoting by $\mathcal{B}$ the unit ball of $\ell_\infty$ which is compact in the $\sigma(\ell_\infty,\ell_1)$ weak-$\star$ topology on $\ell_\infty$, we have $\y^\star_\delta\in \mathcal{B}$. By Banach-Alaoglu's theorem, there is an element $\hat{\y}\in\mathcal{B}$ and a converging subsequence $(\delta_k)_{k\in\N}$ such that
\begin{equation}
\label{conv}
\lim_{k\to\infty}y^\star_{\delta_k,\alpha}\,=\,\hat{y}_\alpha\qquad \forall \alpha\in\N^n.\end{equation}
Let $s\in \N$ be arbitrary, but fixed. By the convergence \eqref{conv} we also have
\[
\lim_{k\to\infty}\M_s(\y^\star_{\delta_k})\,=\,\M_s(\hat{\y})\succcurlyeq0;\quad\lim_{k\to\infty}\M_s(g_j\,\y^\star_{\delta_k})\,=\,\M_s(g_j\,\hat{\y})\,\succcurlyeq0,\:j=1,\ldots,m.
\]
Notice that the subvectors $\y^\star_{d,\delta}=(y_{d,\delta,\alpha}^\star)_{|\alpha|\leq 2d}$ with $\delta=0,1,2,\ldots$ belong to a compact set. Therefore, since $\det(\M_d(\y^\star_{d,\delta}))>0$ for every $\delta$, we also have $\det(\M_d(\hat{\y}))>0$.\par\medskip\noindent
Next, by Putinar's theorem \cite[Theorem 3.8]{lasserre}, $\hat{\y}$ is the sequence of moments of some measure $\hat{\mu}\in \Ms_+(\K)$, and so $\hat{\y}_d=(\hat{y}_\alpha)_{|\alpha|\leq 2d}$ is a feasible solution to \eqref{sdp-ideal}, meaning $\rho\geq\log\det(\M_d(\hat{\y}_d))$. On the other hand, as \eqref{sdp} is a relaxation of \eqref{sdp-ideal}, we have $\rho\leq\rho_{\delta_k}$ for all $\delta_k$. So the convergence \eqref{conv} yields
 \[\rho\leq\,\lim_{k\to\infty}\rho_{\delta_k}\,=\,\log\det \M_d(\hat{\y}_d),\]
 which proves that $\hat{\y}$ is an optimal solution to \eqref{sdp-ideal}, and $\lim_{\delta\to\infty}\rho_\delta=\rho$.

\item As the optimal solution to \eqref{sdp-ideal} is unique, we have $\y^\star=\hat{\y}_d$, and the whole sequence $(\y^\star_{d,\delta})_{\delta\in\N}$ converges to $\y^\star$, that is
\begin{equation}
\label{conv2}
\lim_{\delta\to\infty}y^\star_{\delta,\alpha}\,=\,\hat{y}_\alpha=y^\star_\alpha\qquad \forall |\alpha|\leq 2d.
\end{equation}

\item Finally, to show (c) it suffices to observe that the coefficients of the orthonormal polynomials
$(P_\alpha)_{|\alpha|\leq d}$ with respect to $\y^\star_d$ are continuous functions of the
moments $(y^\star_{\delta,\alpha})_{|\alpha|\leq 2d}$. Therefore, by the convergence \eqref{conv2} 
one has $p^\star_{d,\delta}\to p^\star_d$ where $p^\star_d\in\R[\x]_{2d}$ as in Theorem \ref{th-ideal}.
\item The last point is direct observing that, in this case, the two programs satisfies the same KKT conditions.
\end{enumerate}
\end{proof}

\section{Recovering the measure}
\label{recoverMeasure}

By solving step one as explained in Section \ref{idealProblem}, we obtain a solution $\y^\star_d$
of SDP problem \eqref{sdp}. However  we do not know if $\y^\star_d$ comes from a measure. This would be the case if we can find an atomic measure having these moments and yielding the same value in problem \eqref{sdp}. For this,
we propose two approaches: A first one which follows a procedure by Nie \cite{nie}, and a second one which uses properties of the Christoffel polynomial associated with $\y^\star_d$.

\subsection{Via the method by Nie}
This approach to recovering a measure from its moments is based on a formulation proposed by Nie in \cite{nie}. 

Let $\y_d^\star=(y^\star_{\delta,\alpha})_{|\alpha|\leq 2d}$ be a solution to \eqref{sdp}.  For $r\in\N$ consider the SDP problem
\begin{equation}
\label{sdp-second}
\begin{array}{rl} \displaystyle\min_{\y\in\R^{s(2d+2r)}} & L_\y(f_r)\\
\mbox{s.t.}& \M_{d+r}(\y)\,\succcurlyeq\,0\\
& \M_{d+r-v_j}(g_j\,\y)\,\succcurlyeq\,0,\quad j=1,\ldots,m\\
& y_{\alpha}=y^\star_{\delta,\alpha},\quad |\alpha|\leq 2d,
\end{array}\end{equation}
where $f_r\in\R[\x]_{2(d+r)}$ is a randomly generated polynomial, strictly positive on $\K$, and again $v_j=\lceil d_j/2\rceil$, $j=1,\ldots,m$. 
Then, we check whether the optimal solution $\y^\star_r$ of (\ref{sdp-second}) satisfies the rank condition
\begin{equation}
\label{test}
{\rm rank}\:\M_{d+r}(\y_r^\star)\,=\,{\rm rank}\:\M_{d+r-v}(\y_r^\star),\end{equation}
where $v:=\max_j v_j$. If the test is passed, then we stop, otherwise we repeat with $r:=r+1$. Using linear algebra, we can also extract the points $x_1^\star,\dotsc,x_r^\star\in\mathcal{X}$ which are the support of the representing atomic measure of $\y_d^\star$.
If $\y^\star_d\in \mathcal{M}_{2d}(\K)$, then with probability one, the rank condition (\ref{test}) will be satisfied for a sufficiently large value of $r$.

Experience reveals that in most of the cases it is enough to use the following polynomial
\[\x\mapsto f_r(\x)=\sum_{|\alpha|\leq d+r}\x^{2\alpha}\]
instead of using a random positive polynomial on $\K$. In problem (\ref{sdp-second}) this corresponds to minimizing the trace of $\M_{d+r}(\y)$ (and so induces an optimal solution
$\y$ with low rank matrix $\M_{d+r}(\y)$).

\subsection{Via Christoffel polynomials}
Another possibility to recover the atomic representing measure of $\y_d^\star$ is to find the zeros of the polynomial $p^\star(x)=\binom{n+d}{n}-p^\star_d(x)$, where $p^\star_d$ is the Christoffel polynomial associated with $\y_d^\star$ on $\K$, that i, the set $\{x\in\mathcal{X}:\binom{n+d}{n}-p_d^\star(x)=0\}$. Due to Theorem \ref{th-sdp} this set is the support of the atomic representing measure.

So we minimize the polynomial $p^\star$ on $\mathcal{X}$ and check whether it vanishes on at least $\binom{n+d}{n}$ (and at most $\binom{n+2d}{n}$) points of the boundary of $\K$. That is,
let $p^\star_d$ be as in Theorem \ref{th-sdp} for some fixed $\delta\in\N$ and solve the SDP problem
\begin{equation}
\label{sdp-three}
\begin{array}{rl} \displaystyle\min_{\y\in\R^{s(2(d+r))}}&L_\y(p^\star)\\
\mbox{s.t.}& \M_{d+r}(\y)\,\succcurlyeq\,0,\quad y_0=1,\\
& \M_{d+r-v_j}(g_j\,\y)\,\succcurlyeq\,0,\quad j=1,\ldots,m.
\end{array}
\end{equation}
Since $p^\star_d$ is associated with the optimal solution to \eqref{sdp} for some given $\delta\in\N$, by Theorem \ref{th-sdp}, it satisfies the Putinar certificate \eqref{certificate} of positivity on $\K$. Thus, the value of problem {\ref{sdp-three} is zero for all $r\geq \delta$. Therefore, for every feasible solution $\y_r$ of \eqref{sdp-three} one has $L_{\y_r}(p^\star)\geq0$ (and $L_{\y^\star_d}(p^\star)=0$ for $\y^\star_d$ an optimal solution of \eqref{sdp}).

Alternatively, we can solve the SDP
\begin{equation}
\label{sdp-four}
\begin{array}{rl} \displaystyle\min_{\y\in\R^{s(2d+2r)}} &{\rm trace}(\M_{d+r}(\y))\\
\mbox{s.t.}& L_\y(p^\star)\,=\,0,\\
&\M_{d+r}(\y)\,\succcurlyeq\,0,\quad y_0=1,\\
& \M_{d+r-v_j}(g_j\,\y)\,\succcurlyeq\,0,\quad j=1,\ldots,m.
\end{array}
\end{equation}
We know that  the value of Problem {\ref{sdp-four}} is not greater than ${\rm trace}(\M_d(\hat{\y}^\star))$ where $\hat{\y}^\star$ is an optimal solution to \eqref{sdp-second} for $r=\delta$, because $\hat{\y}^\star$ is feasible for \eqref{sdp-four}.

\subsection{Calculating the corresponding weights}
\label{weights}
After recovering the support $\{x_1,\dotsc,x_\ell\}$ of the atomic representing measure by one of the previously presented methods, we might be interested in also computing the corresponding weights $\omega_1,\dotsc,\omega_\ell$. These can be calculated easily by solving the following linear system of equations: $\sum_{i=1}^\ell \omega_i x_i^\alpha = y_{d,\alpha}^\star$ for all $|\alpha|\leq d$, i.e. $\int_{\mathcal{X}} x^\alpha\mu^\star(dx) = y^\star_{d,\alpha}$.

\section{Examples}
We illustrate the procedure on three examples: a univariate one, a polygon in the plane and one example on the three-dimensional sphere.

All examples are modeled by GloptiPoly 3 \cite{gloptipoly} and YALMIP \cite{yalmip} and solved by MOSEK 7 \cite{mosek} or SeDuMi under the MATLAB R2014a environment. We ran the experiments on an HP EliteBook %840 G1
with 16-GB RAM memory and an Intel Core i5-4300U processor. % under a Windows 7 Professional 64-bit operating system.
We do not report computation times, since they are negligible for our small examples.

\subsection{Univariate unit interval}\label{expl1}
We consider as design space the interval $\mathcal{X}=[-1,1]$ and on it the polynomial measurements $\sum_{j=0}^d \theta_jx^j$ with unknown parameters $\theta\in\R^{d+1}$. To find the optimal design we first solve problem \eqref{sdp}, in other words
\begin{equation*}
\begin{array}{rl}
	\displaystyle\max_{\y} &\log \det \M_d(\y)\\
	\text{s.t.}& \M_{d+\delta}(\y)\,\succcurlyeq\,0,\\
	&\M_{d+\delta-1}(1-\Vert x\Vert^2)\, \y)\,\succcurlyeq\,0,\\
	&y_0=1.
\end{array}\end{equation*}
for given $d$ and $\delta$. For instance, for $d=5$ and $\delta=0$ we obtain the sequence $\y_d^\star\approx\,$(1, 0, 0.56, 0, 0.45, 0, 0.40, 0, 0.37, 0, 0.36).

Then, to recover the corresponding atomic measure from the sequence $\y^\star_d$ we solve the problem
\begin{equation*}
\begin{array}{rl} \displaystyle\min_\y & {\rm trace}\:\M_{d+r}(\y)\\
\mbox{s.t.}& \M_{d+r}(\y)\,\succcurlyeq\,0\\
& \M_{d+r-1}(1-\Vert x\Vert^2)\,\succcurlyeq\,0,\\
& y_{\alpha}=y^\star_{\delta,\alpha},\quad |\alpha|\leq 2d,
\end{array}\end{equation*}
and find the points -1, -0.765, -0.285, 0.285, 0.765 and 1 (for $d=5$, $\delta$=0, $r=1$). As a result, our optimal design is the Dirac measure supported on these points. These match with the known analytic solution to the problem,  which are the critical points of the Legendre polynomial, see e.g. \cite[Theorem 5.5.3, p.162]{dette1997theory}. Calculating the corresponding weights as described in Section \ref{weights}, we find $\omega_1=\dotsb=\omega_6\approx 0.166$.

Alternatively, we compute the roots of  the polynomial
$x\mapsto p^\star(x)=6-p^\star_5(x)$, where $p^\star_5$ is the Christoffel polynomial of degree $2d=10$ on $\mathcal{X}$ and find the same points as in the previous approach by solving problem \eqref{sdp-four}. See Figure \ref{fig:Ch1} for the graph of the Christoffel polynomial of degree 10.

\begin{figure}[ht]
\includegraphics[scale=.5]{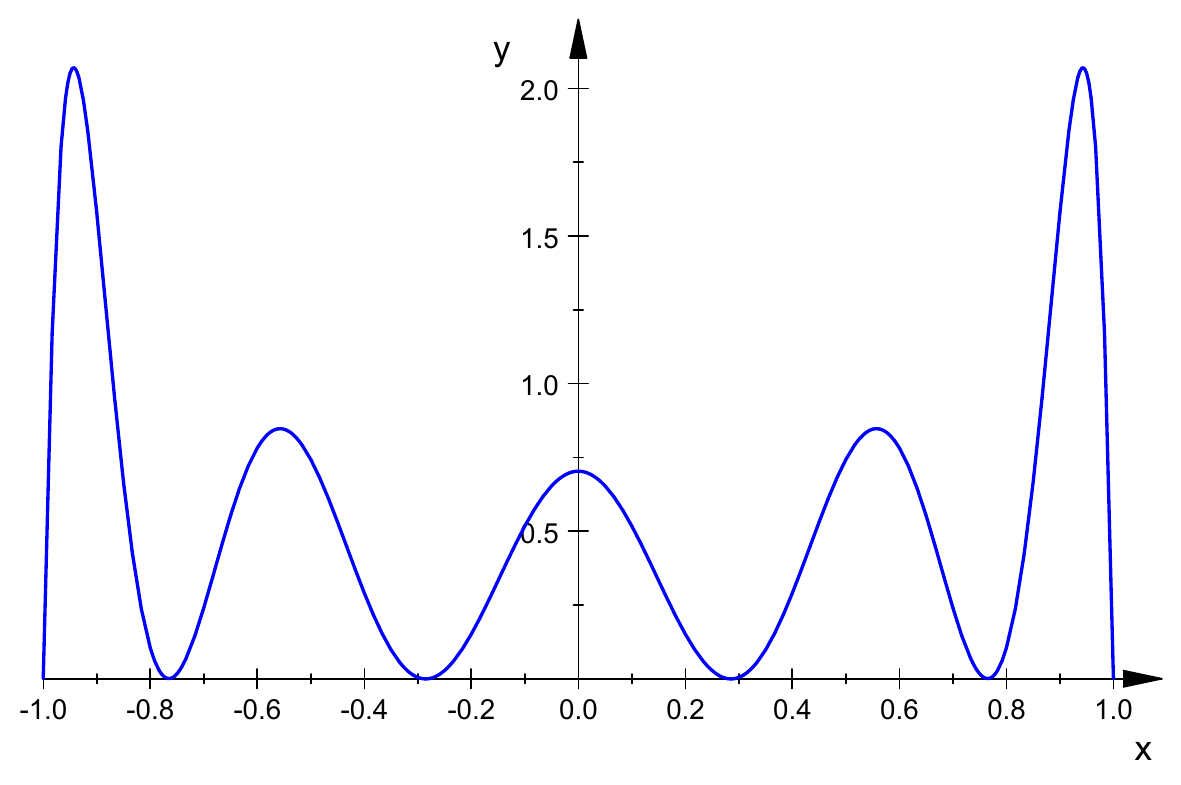}
\caption{Christoffel polynomial for Example \ref{expl1}.}
\label{fig:Ch1}
\end{figure}

We observe that we get less points when using problem \eqref{sdp-three} to recover the support for this example. This may occur due to numerical issues.

\subsection{Wynn's polygon}\label{expl2}
As a two-dimensional example we take the polygon given by the vertices $(-1,-1),\ (-1,1),\ (1,-1)$ and $(2,2)$, scaled to fit the unit circle, i.e. we consider the design space 
\[
\mathcal{X} = \{x\in\R^2 : x_1, x_2 \geq -\tfrac{1}{4}\sqrt{2},\ x_1\leq\tfrac{1}{3}(x_2+\sqrt{2}),\ x_2\leq\tfrac{1}{3}(x_1+\sqrt{2}),\ x_1^2+x_2^2\leq1\}.
\]
Remark that we need the redundant constraint $x_1^2+x_2^2\leq1$ in order to have an algebraic certificate of compactness.

As before, to find the optimal measure for the regression we solve the problems \eqref{sdp} and \eqref{sdp-second}. Let us start by analyzing the results for $d=1$ and $\delta=3$. Solving \eqref{sdp} we obtain $\y^\star\in\R^{45}$ which leads to 4 atoms when solving \eqref{sdp-second} with $r=3$. For the latter the moment matrices of order 2 and 3 both have rank 4, so condition \eqref{test} is fulfilled. As expected, the 4 atoms are exactly the vertices of the polygon.

Again, we could also solve problem \eqref{sdp-four} instead of \eqref{sdp-second} to receive the same atoms. As in the univariate example we get less points when using problem \eqref{sdp-three}. To be precise, GloptiPoly is not able to extract any solutions for this example.

When increasing $d$, we get an optimal measure with a larger support. For $d=2$ we recover 7 points, and 13 for $d=3$. See Figure \ref{fig:wynn} for the polygon, the supporting points of the optimal measure and the 
$\binom{2+d}{2}$-level set of the Christoffel polynomial $p^\star_d$ for different $d$. The latter demonstrates graphically that the set of zeros of $\binom{2+d}{d}-p^\star_d$  intersected with $\mathcal{X}$ are indeed the atoms of our representing measure. In Figure \ref{fig:wynnweights} we visualized the weights corresponding to each point of the support for the different $d$.

\begin{figure}[ht]
\includegraphics[width=.31\textwidth]{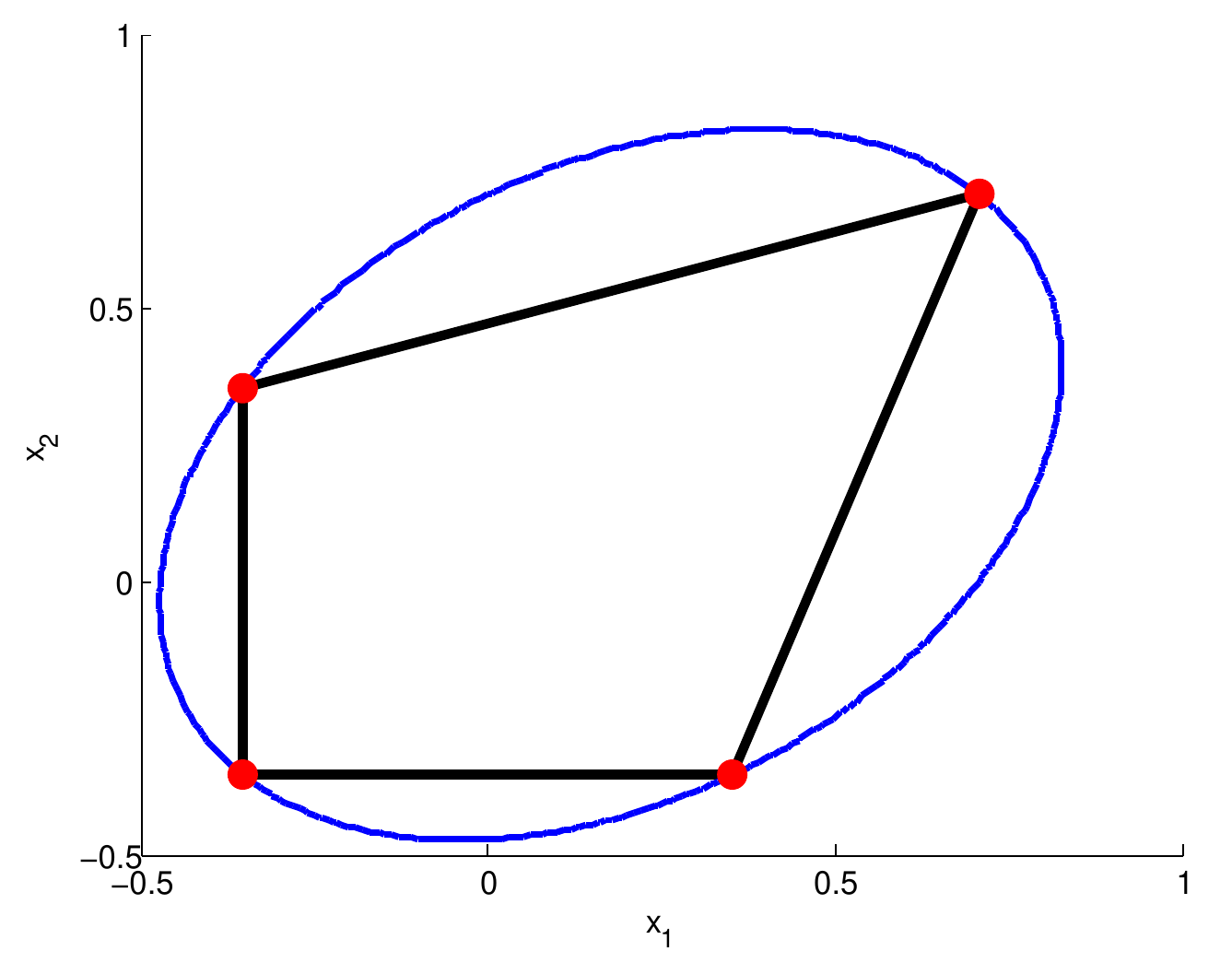}
\includegraphics[width=.31\textwidth]{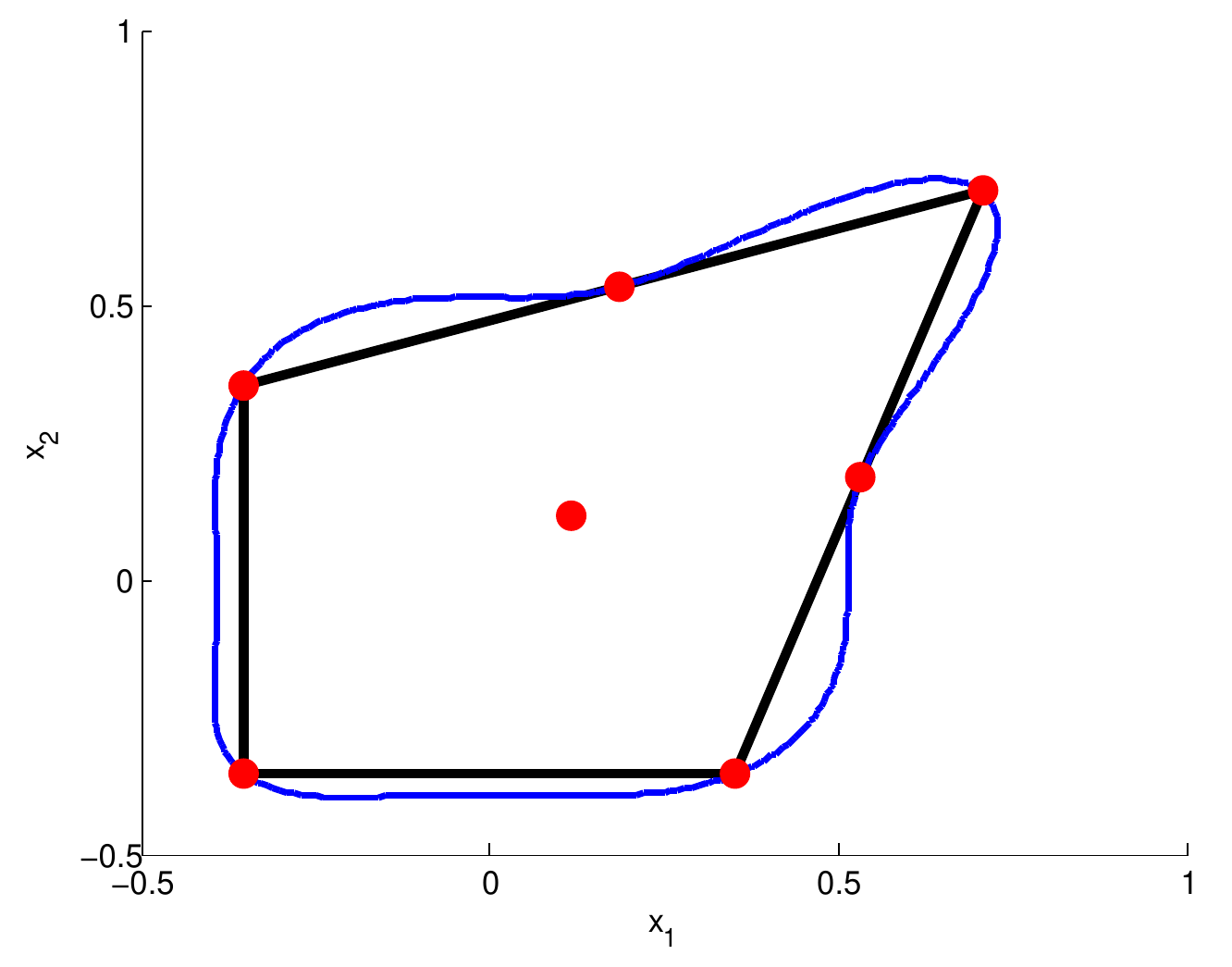}
\includegraphics[width=.31\textwidth]{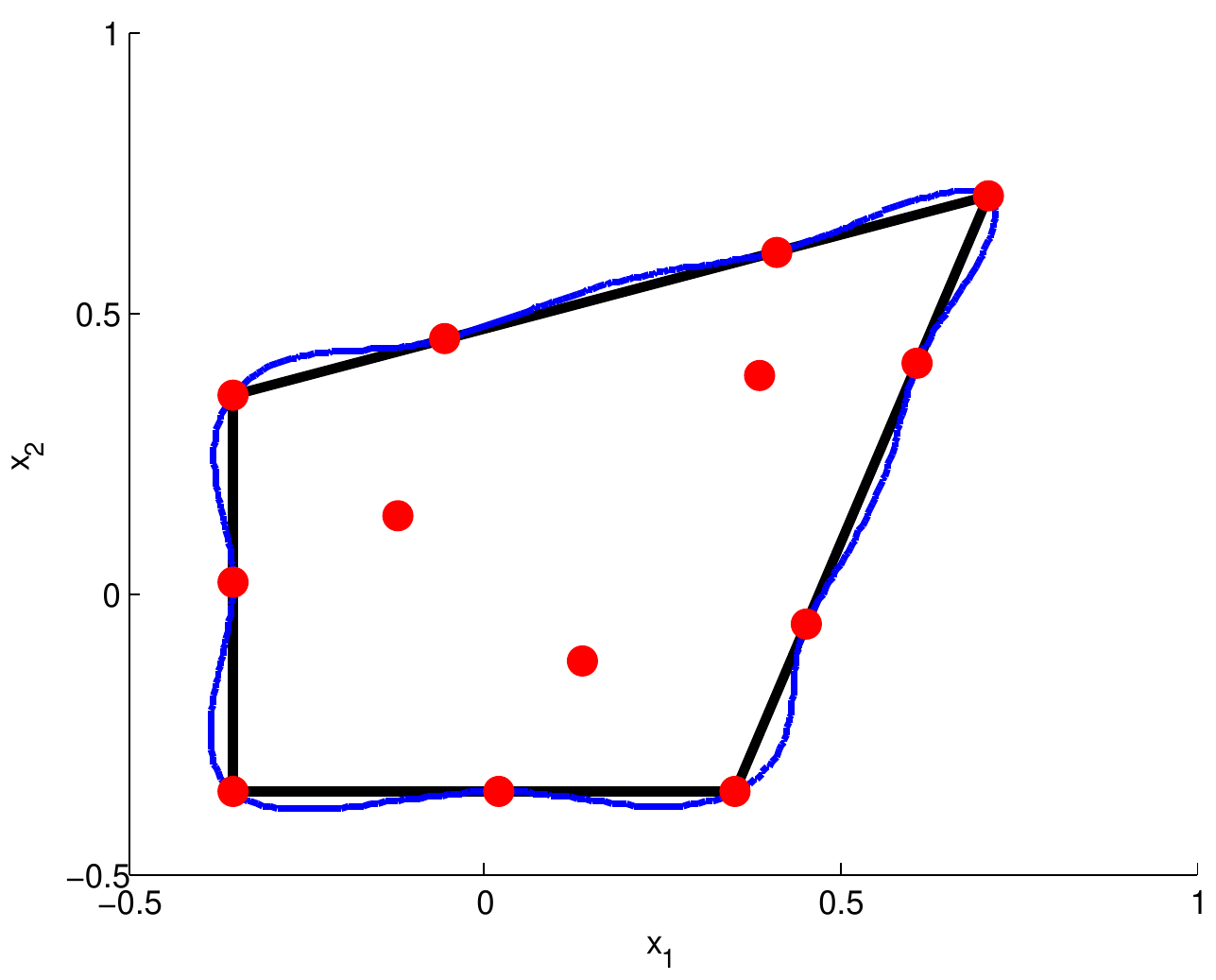}
\caption{The polygon (bold black) of Example \ref{expl2}, the support of the optimal design measure (red points) and the zero level set of the Christoffel polynomial (thin blue) for $d=1$ (left), $d=2$ (middle), $d=3$ (right) and $\delta=3$.}
\label{fig:wynn}
\end{figure}

\begin{figure}[ht]
\includegraphics[width=.31\textwidth]{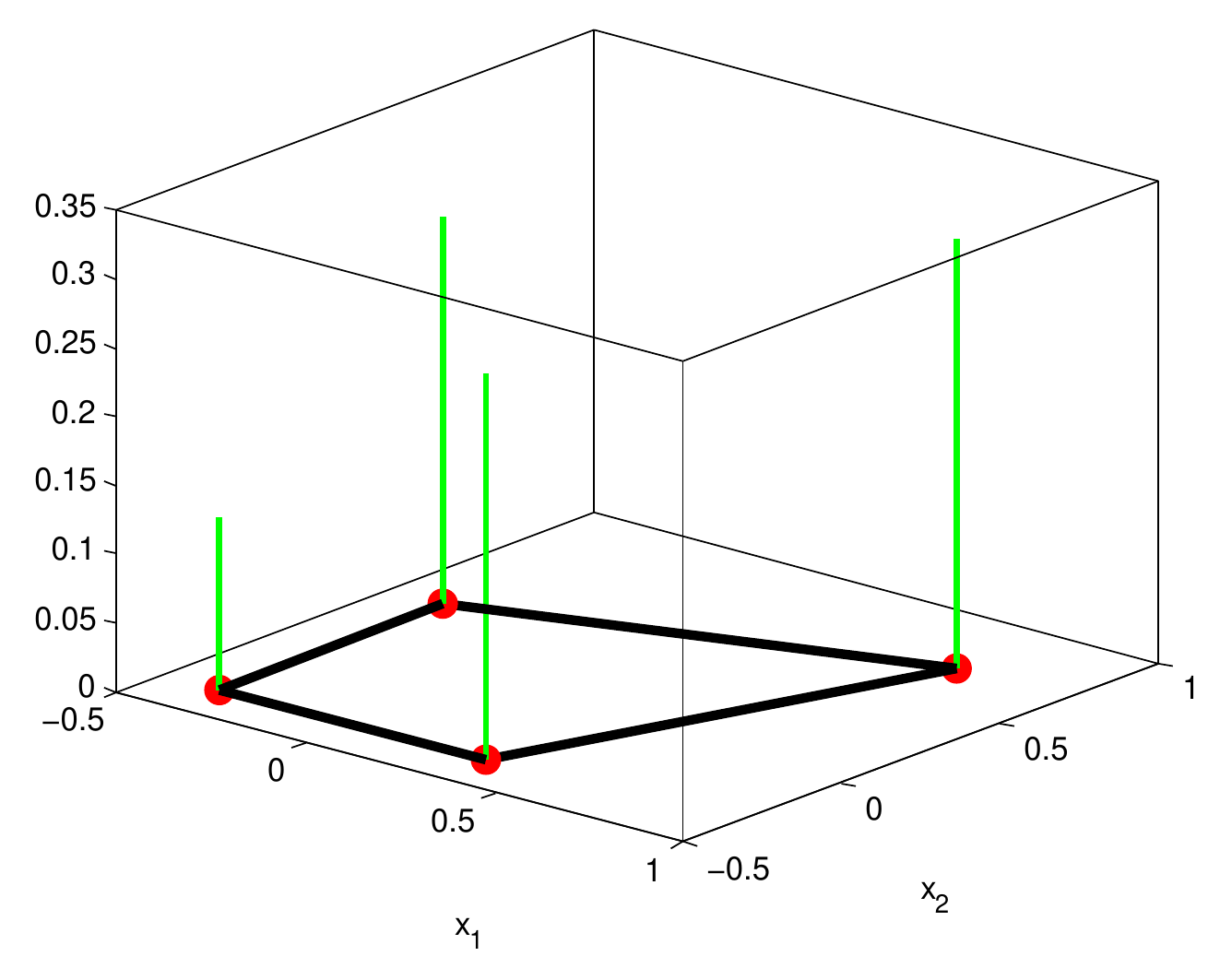}
\includegraphics[width=.31\textwidth]{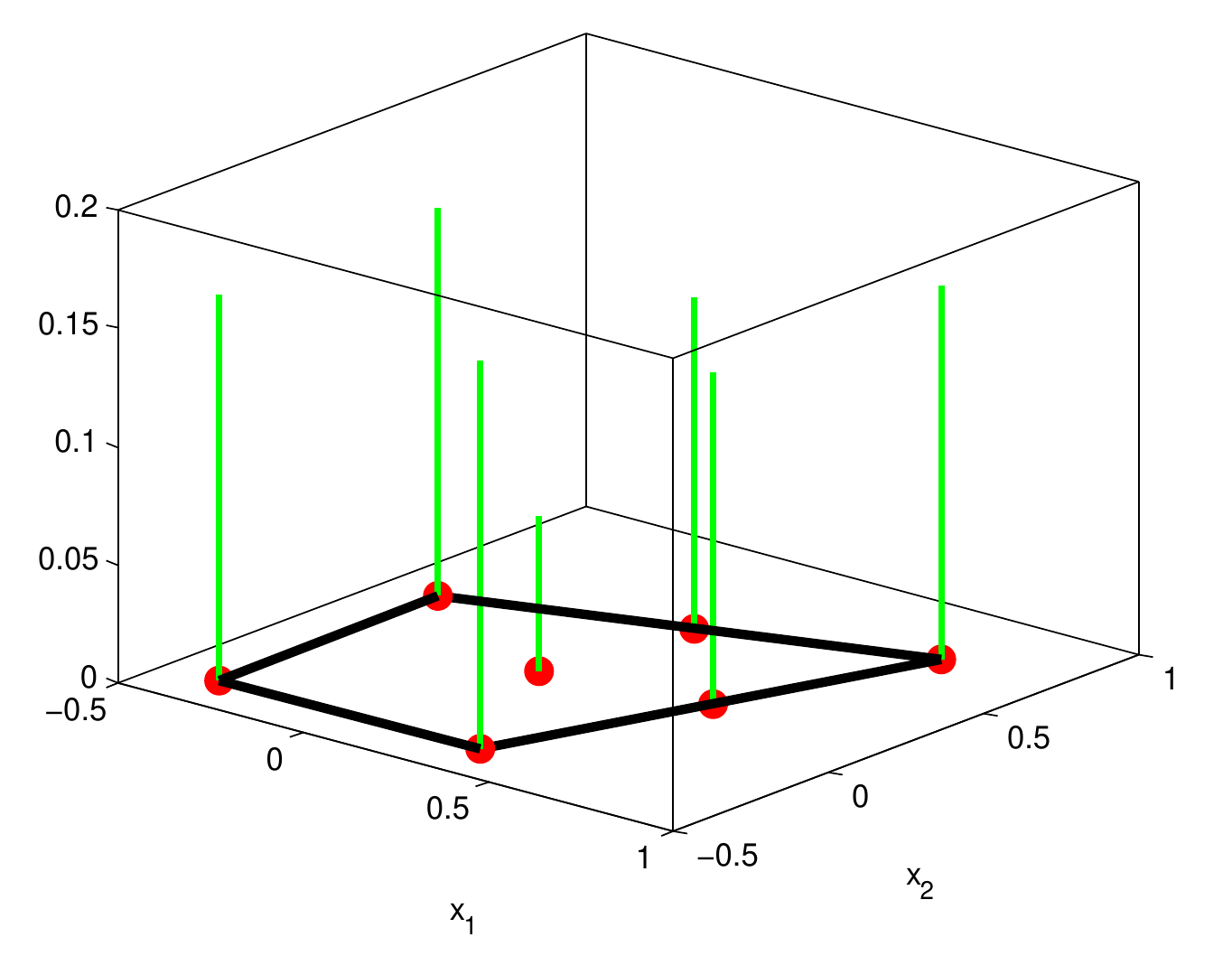}
\includegraphics[width=.31\textwidth]{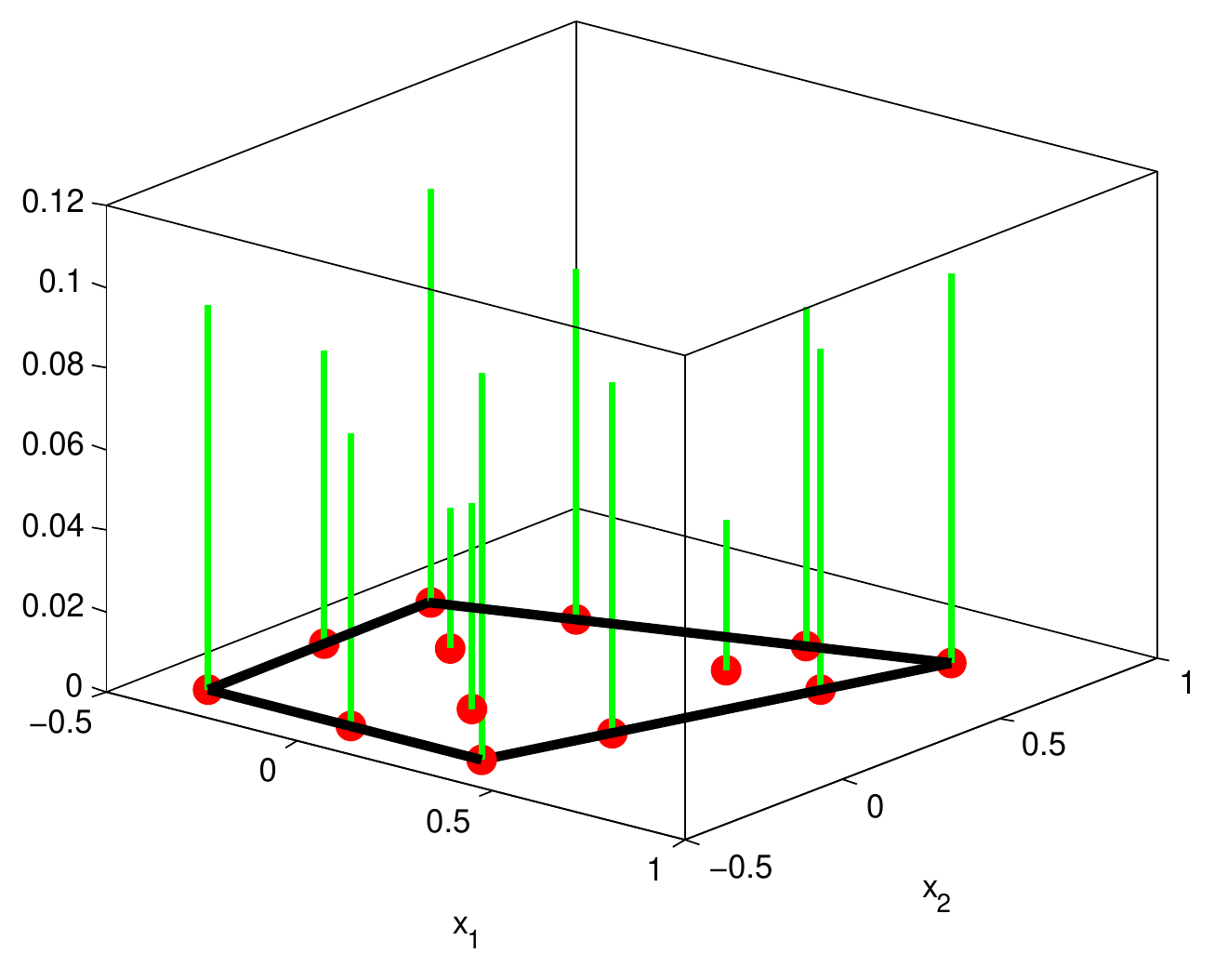}
\caption{The polygon (bold black) of Example \ref{expl2} and the support of the optimal design measure (red points) with the corresponding weights (green bars) for $d=1$ (left), $d=2$ (middle), $d=3$ (right) and $\delta=3$.}
\label{fig:wynnweights}
\end{figure}

\subsection{The 3-dimensional unit sphere}\label{expl3}
Last, let us consider the regression for the degree $d$ polynomial measurements $\sum_{|\alpha|\leq d} \theta_\alpha x^\alpha$ on the unit sphere $\mathcal{X}=\{x\in\R^3: x_1^2+x_2^2+x_3^2=1\}$. As before, we first solve problem \eqref{sdp}. For $d=1$ and $\delta\geq0$ we obtain the sequence $\y^\star_d\in\R^{10}$ with $y_{000}^\star=1,\ y_{200}^\star=y_{020}^\star=y_{002}^\star=0.333$ and all other entries zero.

In the second step we solve problem \eqref{sdp-second} to recover the measure. For $r=2$ the moment matrices of order 2 and 3 both have rank 6, meaning the rank condition \eqref{test} is fulfilled, and we obtain the six atoms $\{(\pm1,0,0),(0,\pm1,0),(0,0,\pm1)\}\subseteq\mathcal{X}$ on which the optimal measure $\mu\in\Ms_+(\mathcal{X})$ is uniformly supported.

For quadratic regressions, i.e. $d=2$, we obtain an optimal measure supported on 14 atoms evenly distributed on the sphere. Choosing $d=3$, meaning cubic regressions, we find a Dirac measure supported on 26 points which again are evenly distributed on the sphere. See Figure \ref{fig:sphere} for an illustration of the supporting points of the optimal measures for $d=1$, $d=2$, $d=3$ and $\delta=0$.

\begin{figure}[ht]
\includegraphics[width=.3\textwidth]{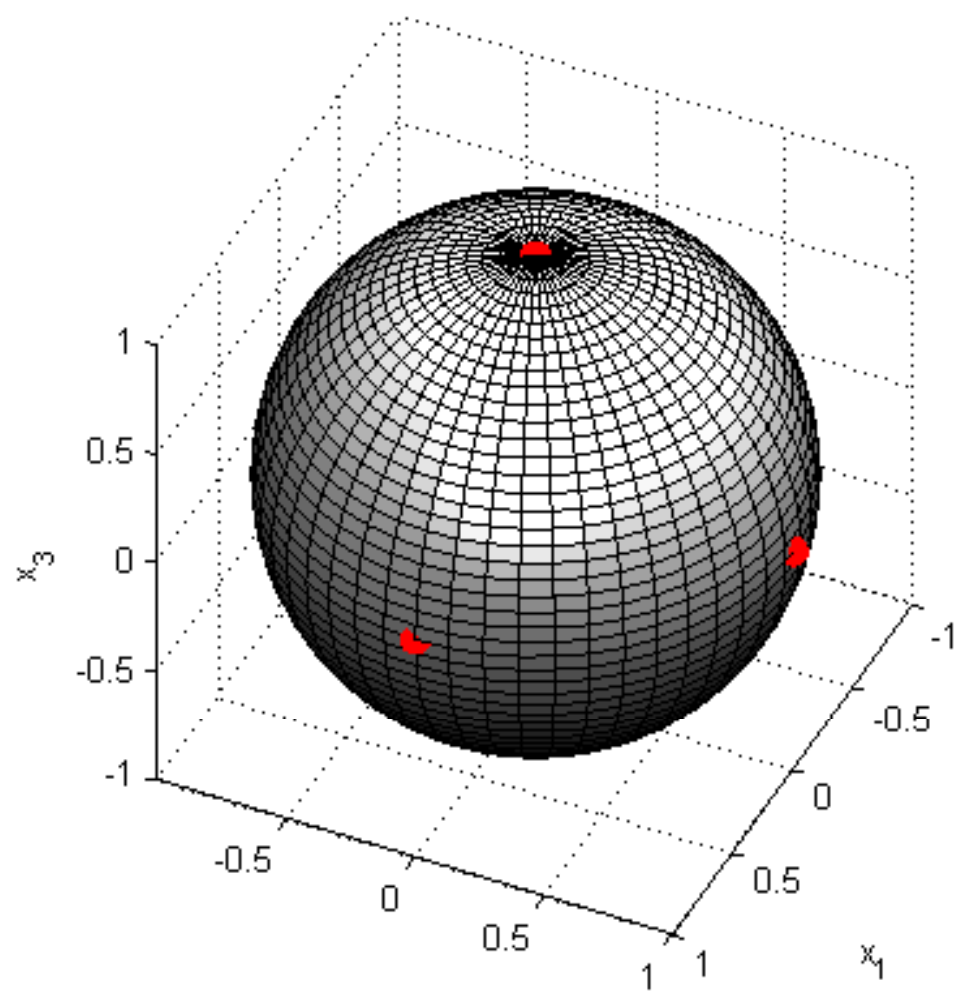}
\includegraphics[width=.3\textwidth]{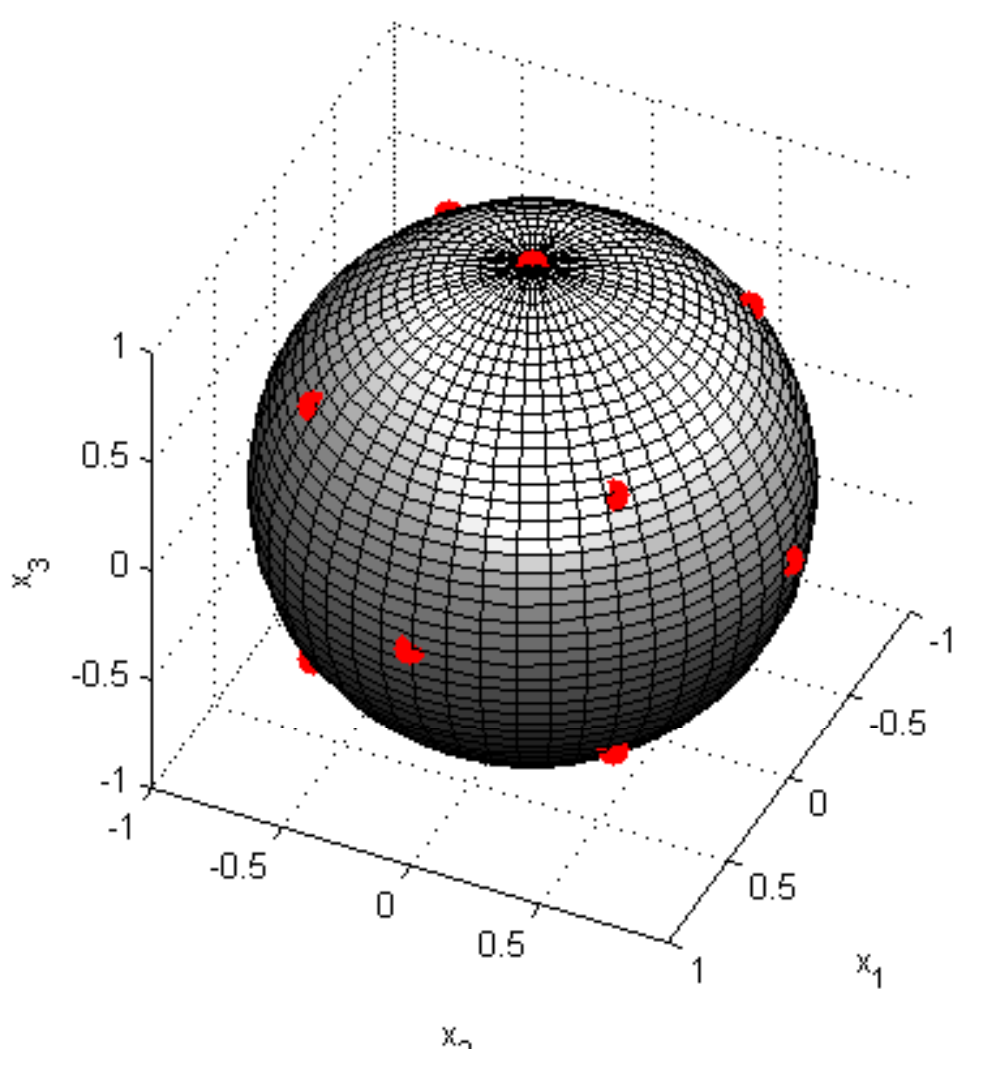}
\includegraphics[width=.3\textwidth]{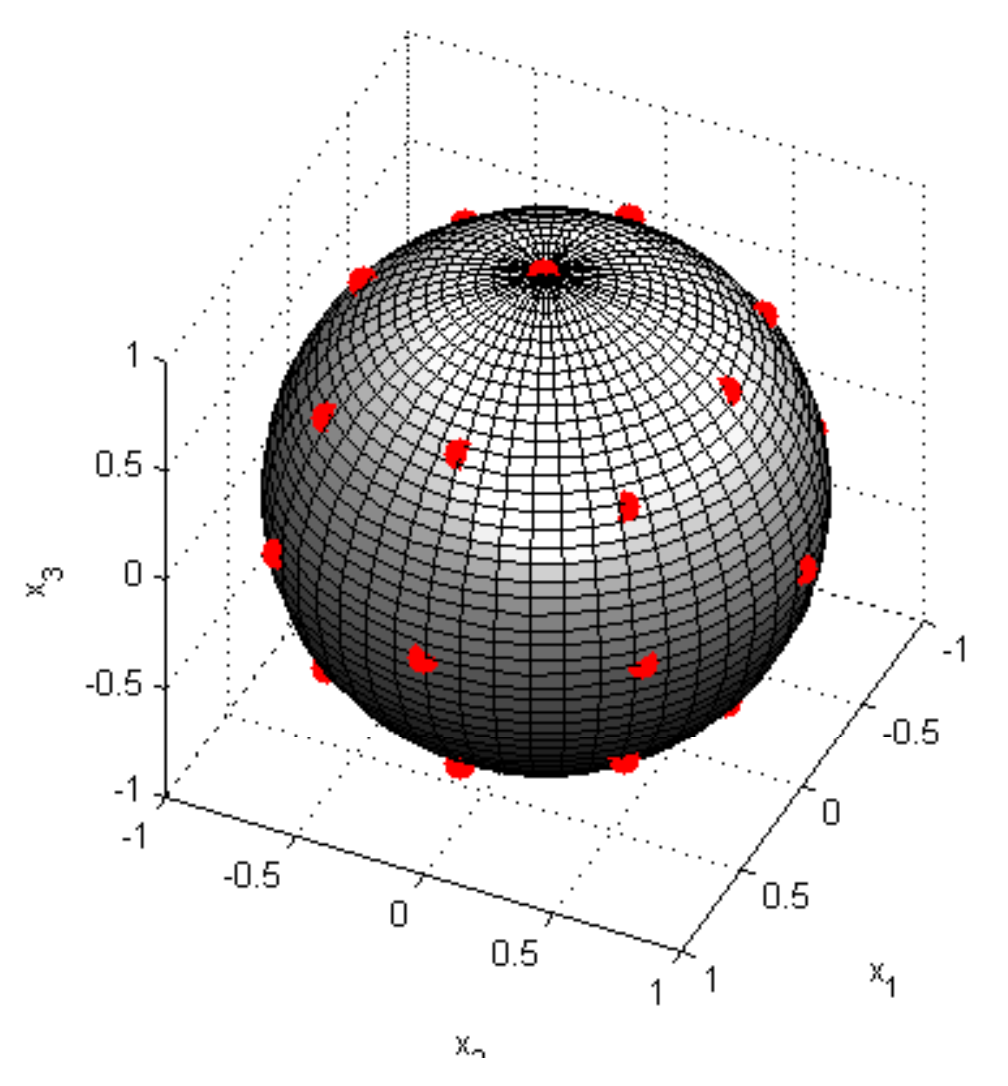}
\caption{The red points illustrate the support of the optimal design measure for $d=1$ (left), $d=2$ (middle), $d=3$ (right) and $\delta=0$ for Example \ref{expl3}.}
\label{fig:sphere}
\end{figure}

Using the method via Christoffel polynomials gives again less points. No solution is extracted when solving problem \eqref{sdp-four} and we find only two supporting points for problem \eqref{sdp-three}.

\subsection*{Acknowledgments}
We thank Henry Wynn for communicating the polygon of Example \ref{expl2} to us.
Feedback from Luc Pronzato, Lieven Vandenberghe and Weng Kee Wong was also appreciated.

\bibliographystyle{amsalpha}

\end{document}